\documentclass[letterpaper, 11pt,  reqno]{amsart}

\usepackage{amsmath,amssymb,amscd,amsthm,amsxtra, esint}

\usepackage[implicit=true]{hyperref}

\allowdisplaybreaks[2]

\usepackage{color}

\definecolor{gr}{rgb}   {0.,   0.69,   0.23 }
\definecolor{bl}{rgb}   {0.,   0.5,   1. }
\definecolor{mg}{rgb}   {0.85,  0.,    0.85}
\definecolor{yl}{rgb}   {0.8,  0.7,   0.}
\definecolor{or}{rgb}  {0.7,0.2,0.2}

\newtheorem{theorem}{Theorem} [section]

\newtheorem{lemma}[theorem]{Lemma}
\newtheorem{proposition}[theorem]{Proposition}
\newtheorem{remark}[theorem]{Remark}

\newtheorem{definition}[theorem]{Definition}

\DeclareMathOperator*{\supp}{supp}

\DeclareMathOperator{\tr}{tr}

\newcommand{\nn}{|\hspace{-0.4mm}|\hspace{-0.4mm}|}

\newcommand{\noi}{\noindent}
\newcommand{\Z}{\mathbb{Z}}
\newcommand{\R}{\mathbb{R}}
\newcommand{\C}{\mathbb{C}}
\newcommand{\T}{\mathbb{T}}

\let\Re=\undefined\DeclareMathOperator*{\Re}{Re}
\let\Im=\undefined\DeclareMathOperator*{\Im}{Im}

\newcommand{\M}{\mathcal{M}}

\newcommand{\NB}{\mathbb{N}}

\newcommand{\FL}{\mathcal{F}L} 

\renewcommand{\S}{\mathcal{S}}

\newcommand{\F}{\mathcal{F}}

\newcommand{\al}{\alpha}
\newcommand{\be}{\beta}
\newcommand{\dl}{\delta}

\newcommand{\eps}{\varepsilon}
\newcommand{\kk}{\kappa}

\newcommand{\ld}{\lambda}

\newcommand{\ft}{\widehat}

\newcommand{\cj}{\overline}
\newcommand{\dx}{\partial_x}
\newcommand{\dt}{\partial_t}
\newcommand{\dd}{\partial}

\newcommand{\ta}{\theta}

\renewcommand{\l}{\ell}

\newcommand{\les}{\lesssim}
\newcommand{\ges}{\gtrsim}

\newcommand{\jb}[1]
{\langle #1 \rangle}

\newcommand{\ind}{\mathbf 1}

\newcommand{\Gf}{\mathfrak{G}}

\numberwithin{equation}{section}
\numberwithin{theorem}{section}

\newcommand{\MH}{\textit MH}
\newcommand{\If}{\mathfrak{I}}

\newtheorem*{ackno}{Acknowledgements}

\newcommand{\ceil}[1]
{\lceil #1 \rceil}

\begin{document}

\baselineskip = 14pt

\title
[GWP of  the 1-$d$ cubic NLS in almost critical spaces]
{Global well-posedness of the one-dimensional cubic nonlinear Schr\"odinger equation
in almost critical spaces}

\author[T.~Oh and  Y.~Wang]
{Tadahiro Oh and Yuzhao Wang}

\address{
Tadahiro Oh, School of Mathematics\\
The University of Edinburgh\\
and The Maxwell Institute for the Mathematical Sciences\\
James Clerk Maxwell Building\\
The King's Buildings\\
Peter Guthrie Tait Road\\
Edinburgh\\ 
EH9 3FD\\
 United Kingdom}

\email{hiro.oh@ed.ac.uk}

\address{
Yuzhao Wang\\
School of Mathematics\\
Watson Building\\
University of Birmingham\\
Edgbaston\\
Birmingham\\
B15 2TT\\ United Kingdom}

\email{y.wang.14@bham.ac.uk}



\begin{abstract}
In this paper, we first introduce 
a new function space  $\MH^{\ta, p}$ 
whose norm is given by the $\l^p$-sum
of modulated $H^\ta$-norms of a given function.
In particular, when $\ta < -\frac 12$, 
 we show that the
space $\MH^{\ta, p}$ agrees with 
the modulation space
$M^{2, p}(\R)$ on the real line
and the Fourier-Lebesgue space $\FL^{p}(\T)$
on the circle.
We use this equivalence of the norms
and the Galilean symmetry to adapt
the conserved quantities constructed by Killip-Vi\c{s}an-Zhang
to the modulation space 
and Fourier-Lebesgue space
setting.
By applying the scaling symmetry, 
we then prove global well-posedness
of the one-dimensional 
cubic nonlinear Schr\"odinger equation (NLS)
in almost critical spaces.
More precisely, we show that the cubic NLS on $\R$
is globally well-posed in $M^{2, p}(\R)$
for any $p < \infty$, 
while 
the renormalized cubic NLS on $\T$
is globally well-posed in $\FL^p(\T)$
for any $p < \infty$.

In Appendix, 
we also 
establish analogous global-in-time bounds for 
 the modified KdV equation (mKdV)
in the modulation spaces on the real line and 
in the Fourier-Lebesgue spaces on the circle.
An additional key ingredient of the proof in this case is 
 a Galilean transform which  converts
the mKdV to the mKdV-NLS equation.

\end{abstract}

\maketitle

\tableofcontents

\section{Introduction}

\subsection{One-dimensional cubic nonlinear Schr\"odinger equation}
In this paper, we study the following Cauchy problem
of the one-dimensional cubic nonlinear Schr\"odinger equation (NLS) on $\M = \R$ or $\T = \R/(2\pi\Z)$:
\begin{align}
\begin{cases}
i \dt u = \dx^2 u \mp 2 |u|^2 u, \\
u|_{t= 0} = u_0.
\end{cases}
\label{NLS1}
\end{align}

\noi
The equation \eqref{NLS1}
 arises in various physical
situations for the description of wave propagation in nonlinear optics, fluids, and plasmas;
 see \cite{SULEM} for a general review.
 It is also known to be one of the simplest partial differential equations (PDEs) with complete integrability
 \cite{ZS, AKNS, A2, GK}.
Our main goal in this paper is to exploit the complete integrable structure
of the equation and prove global well-posedness of \eqref{NLS1} in almost critical spaces.

The Cauchy problem  \eqref{NLS1} has been studied extensively
by many mathematicians.
Tsutsumi \cite{Tsutsumi} and Bourgain \cite{BO1}
proved global well-posedness of \eqref{NLS1}
in $L^2(\M)$ with $\M = \R$ and $\T$, respectively.
Before going over the known  results for \eqref{NLS1} below $L^2(\M)$, 
let us first recall two important symmetries
that \eqref{NLS1} enjoys.
The scaling symmetry states that if $u(x, t)$ is a solution to \eqref{NLS1}
on $\M = \R$ with initial data $u_0$, 
then the $\ld$-scaled function 
\begin{align}
u(x, t) \longmapsto
u_\ld(x, t) = \ld^{-1} u (\ld^{-1}x, \ld^{-2}t)
\label{scaling}
\end{align}

\noi
is also a solution to \eqref{NLS1} with the $\ld$-scaled initial data 
$u_{0, \ld}(x) = \ld^{-1} u_0 (\ld^{-1}x)$.
Associated to this scaling symmetry, 
there is a  scaling-critical Sobolev regularity $s_\textup{crit}$
such that the homogeneous $\dot{H}^{s_\text{crit}}$-norm is invariant
under the scaling symmetry.
In the case of the one-dimensional cubic NLS \eqref{NLS1}, 
 the scaling-critical Sobolev regularity is $s_\text{crit} = -\frac{1}{2}$
 and it is known that \eqref{NLS1} is ill-posed
 in $H^s(\M)$ for $s \leq s_\text{crit} = -\frac 12$
 with  $\M = \R$ or $\T$
 in the sense of norm inflation 
 \cite{CCT2b, Kishimoto, OW, O1}; 
given any $\eps > 0$, 
there exist a solution $u$ to \eqref{NLS1} on $\M$
and $t_\eps  \in (0, \eps) $ such that 
\begin{align}
 \| u(0)\|_{H^s(\M)} < \eps \qquad \text{ and } \qquad \| u(t_\eps)\|_{H^s(\M)} > \eps^{-1}.
\label{NI}
 \end{align}

\noi
Note that 
this is a stronger notion
of ill-posedness than  the failure of continuity of the solution map at $u_0 \equiv 0$.
The other symmetry 
of importance here is 
the Galilean symmetry;  if $u(x, t)$ is a solution to \eqref{NLS1}
on $\R$ with initial condition $u_0$, then
\begin{align}
u^\beta(x, t) = \mathcal{G}_\be(u)(x, t) := e^{- i\be x}e^{i\be^2 t}  u (x- 2\beta t, t)
\label{Galilei}
\end{align}

\noi
is also a solution to \eqref{NLS1} with the modulated initial condition 
$u_0^\beta(x) = e^{-i\be x}  u_0 (x)$.
On the Fourier side, the Galilean symmetry is expressed as
\begin{align*}
\ft{u^\be}(\xi, t) = 
e^{- i\be^2 t}  e^{- 2 i \be  \xi t}\ft u (\xi + \be, t), 
\end{align*}

\noi
basically corresponding to a translation in frequencies.
In particular, we need to impose $\be \in   \Z$
on the circle $\M = \T$.
We point out that 
the Galilean symmetry 
induces another critical regularity $s_\text{crit}^\infty  =  0$
since it  preserves the   $L^2$-norm.
In fact, there is a dichotomy between the behavior of solutions to \eqref{NLS1}
in $L^2(\M)$ and in negative Sobolev spaces.
On the one hand, \eqref{NLS1} is known to be well-posed in $L^2(\M)$.
On the other hand, it is known to be mildly ill-posed
in negative Sobolev spaces
in the sense of the failure of local uniform continuity 
of the solution map:
 $\Phi(t): u_0\in H^s(\M) \mapsto u(t) \in H^s(\M)$;
 see \cite{KPV, BGT, CCT1}.
Moreover, on the circle, \eqref{NLS1} is known to be ill-posed
below $L^2(\T)$; see \cite{CCT2, MOLI, GO}.
In particular, in \cite{GO}, 
the first author (with Z.\,Guo) showed 
non-existence of solutions to \eqref{NLS1} on $\T$ with initial data outside $L^2(\T)$.
This last result was proved  by first establishing an existence result
for the following renormalized cubic NLS on $\T$:
\begin{align}
\begin{cases}
i \dt u = \dx^2 u \mp 2 ( |u|^2  - 2\fint_\T |u|^2 dx)u \\
u|_{t= 0} = u_0,
\end{cases}
\qquad (x, t) \in \T\times \R, 
\label{NLS2}
\end{align}

\noi
where
 $\fint f(x) dx := \frac{1}{2\pi} \int_\T f(x)  dx.$
 In fact, it is known that, 
 while it is equivalent to the standard cubic NLS \eqref{NLS1} within $L^2(\T)$, 
  this renormalized cubic NLS \eqref{NLS2} on $\T$
behaves better outside $L^2(\T)$
 and in fact share many common properties with the cubic NLS on $\R$
 outside $L^2$; see a survey paper \cite{OS}.
Note that 
there is 
  no ill-posedness result below $L^2(\M)$
for the cubic NLS \eqref{NLS1} on $\R$
or the renormalized cubic NLS \eqref{NLS2} on $\T$, contradicting either existence, uniqueness, or continuous dependence, 
and the well-posedness of \eqref{NLS1} on $\R$ 
and \eqref{NLS2} on $\T$ in negative Sobolev spaces 
(in particular uniqueness) has been a long-standing challenging open question in the field.
In \cite{CCT3, KT1, KT2}, 
Christ-Colliander-Tao and Koch-Tataru 
independently proved  existence (without uniqueness)
of solutions to  the cubic NLS \eqref{NLS1} on $\R$ in negative Sobolev spaces.
An analogous existence result for the renormalized cubic NLS \eqref{NLS2} on $\T$
was established in \cite{GO}.
More recently, 
Koch-Tataru \cite{KT3} and Killip-Vi\c{s}an-Zhang \cite{KVZ}
exploited the complete integrable structure of the equation
and proved global-in-time a priori bounds
on the $H^s$-norm of solutions
in the scaling-subcritical range: $s > -\frac 12$.
In the following, we combine the result in \cite{KVZ}
with the scaling and Galilean symmetries
to prove global well-posedness of the cubic NLS \eqref{NLS1} on $\R$
and the renormalized cubic NLS \eqref{NLS2} on $\T$
in almost critical spaces with respect to the scaling symmetry.

\subsection{Fourier-Lebesgue spaces
and modulation spaces}
In this subsection, we first recall the definitions of 
the Fourier-Lebesgue spaces
and the modulation spaces.
Then, we go over the known well-posedness results for the cubic NLS \eqref{NLS1} on $\R$
and the renormalized cubic NLS \eqref{NLS2} on $\T$ in these spaces.
Lastly, we introduce a new function space $\MH^{\ta, p}$ and 
show that this space coincides with 
 the modulation spaces on $\R$
and the Fourier-Lebesgue spaces on $\T$
in a certain regime (Lemma \ref{LEM:equiv}).
This equivalence of the norms will be a key ingredient 
for the proof of the main result (Theorem \ref{THM:1}).

Our conventions for the Fourier transform are as follows:
\[ \ft f(\xi) = \frac{1}{\sqrt{2\pi}} \int_\R f(x) e^{-ix \xi} dx
\qquad\text{and}\qquad
 f(x) = \frac{1}{\sqrt{2\pi}} \int_\R \ft f(\xi) e^{ix \xi} d\xi
\]

\noi
for functions on the real line $\R$
and 
\[ \ft f(\xi) = \frac{1}{\sqrt{2\pi}} \int_0^{2\pi} f(x) e^{-ix \xi} dx
\qquad\text{and}\qquad
 f(x) =\frac{1}{\sqrt{2\pi}} \sum_{\xi \in  \Z} \ft f(\xi) e^{ix \xi}
\]

\noi
for functions on the circle $\T$ (with $\xi \in \Z$).
Given $\M = \R$ or $\T$, 
let   $\ft \M$ denote the Pontryagin dual of $\M$, 
i.e.~
\begin{align*}
\ft \M = \begin{cases}
 \R & \text{if }\M = \R, \\
 \Z & \text{if } \M = \T.
\end{cases}
\end{align*}

\noi
When $\ft \M = \Z$, 
we endow it with the counting measure.
Given $s \in \R$ and $1 \leq p \leq \infty$, we define
the Fourier-Lebesgue space $\F L^{s, p}(\M)$  by the norm:
\begin{align*}
\|f \|_{\F L^{s, p}(\M)} = \| \jb{\xi}^s \ft f(\xi) \|_{L^{p}_\xi(\ft \M)}
\end{align*}

\noi
with the usual modification when $p = \infty$.
Here,  $\jb{\,\cdot\,} = (1 + |\cdot|^2)^\frac{1}{2}$.
When $s = 0$, we simply set $\FL^p(\M) = \FL^{0, p}(\M)$.
Note that we have $\FL^p (\T) \supset L^2(\T)$ on the circle for $p \geq 2$.

Next, we  recall  the definition of the modulation spaces $M^{r, p}_s(\R)$ on the real line; see \cite{FG1, FG2}. 
Let $\psi \in \S(\R)$ such that
\begin{align*}
\supp \psi \subset [-1, 1]
\qquad \text{and} \qquad \sum_{k \in \Z} \psi(\xi -k) \equiv 1.
\end{align*}

\noi
Then, the modulation space $M_s^{r, p}(\R)$ is defined
as the collection of all tempered distributions
$f\in\S'(\R)$ such that
$\|f\|_{M_s^{r, p}}<\infty$, where
the $M_s^{r, p}$-norm is defined by
\begin{equation} \label{mod}
\|f\|_{M_s^{r, p} (\R)} 
= \big\| \jb{n}^s\|\psi_n(D) f \|_{L_x^r(\R)} \big\|_{\l^p_n(\Z)}.
\end{equation}

\noi
Here, 
$\psi_n(D)$ is the Fourier multiplier operator
with the multiplier 
\begin{align}
\psi_n(\xi) := \psi(\xi - n).
\label{psi1}
\end{align}

\noi
When $s = 0$, we simply set 
$M^{r, p} (\R) =  M_0^{r, p} (\R)$.
In the following, we only consider $r = 2$.
In this case, we have
\begin{align*}
M^{2, p}(\R) \supset \FL^p(\R)
\end{align*}

\noi
for $p\geq 2$.

\begin{remark}\rm
The modulation spaces $M^{r, p}$ have an equivalent characterization via the short-time (or windowed) Fourier transform (STFT).
Given
a non-zero window function $\phi\in \S(\R)$,
we define  the STFT $V_\phi f$ of a tempered distribution $f\in \S'(\R)$ with respect to $\phi$ by
\[V_\phi f(x, \xi)=\frac{1}{\sqrt{2\pi}} \int_{\R}f(y)\cj{\phi(y-x)}e^{- iy \xi}\,dy.\]

\noi
Then, we have the equivalence of norms:
\begin{align}
\|f\|_{M^{r, p}}\sim_\phi 
\nn f \nn_{M^{r, p}}
:= \big\|\|V_\phi f
\|_{L_x^r} \big\|_{L^p_\xi},
\label{mod2}
\end{align}
where the implicit constants depend on the window function $\phi$.
In view of the definition \eqref{mod2} of the $\nn\cdot\nn_{M^{r, p}}$-norm, 
it may be tempting to consider this norm on $\T$.
It is, however, known that, for $1 \leq r, p \leq \infty$,  we have
\begin{align}
M^{r, p}(\T) = \FL^p(\T)
\label{mod5}
\end{align}

\noi
on the circle; see \cite{RSTT}.

\end{remark}

Let us now discuss critical regularities for \eqref{NLS1} and \eqref{NLS2}
in the context of the Fourier-Lebesgue spaces
and the modulation spaces.
A direct computation shows that 
the homogeneous Fourier-Lebesgue space $\dot \FL^{s, p}(\R)$
is invariant under
the scaling symmetry \eqref{scaling}
when $s = s_\text{crit}(p) = - \frac 1p$
with the understanding $s_\text{crit} (\infty) = 0$.
In particular, when $s = 0$, the cubic NLS \eqref{NLS1} on $\R$ is scaling-critical
in $\FL^\infty(\R)$.
While there is no scaling symmetry on the circle, 
we say that the renormalized cubic NLS \eqref{NLS2} on $\T$ is scaling-critical in $\FL^\infty(\T)$.
On the other hand, 
 the modulation spaces
 are based on the unit cube decomposition of the frequency space
 and thus there is no scaling for the modulation spaces.\footnote{See \cite{BOmod} for modulation spaces adapted to scaling.}
At the same time, 
a change of variables and interpolating the $r = 2$ and $r = \infty$ cases yield
 the following bound:
\begin{align}
\| f_\ld \|_{\dot M^{r, p}_s} \les \ld^{ - s - \frac 1p} \| f \|_{\dot M^{r, p}_s}
\label{scaling2}
\end{align}

\noi
for any $\ld \geq 1$, provided that $p \geq r'$ and $r \geq 2$.
See also \cite{ST} for a further discussion on the scaling properties
of the modulation spaces.
This shows that  
$s = s_\text{crit}(p) = - \frac 1p$ is (essentially) a scaling-critical regularity 
for \eqref{NLS1} in terms of the modulation spaces
$M^{r, p}_s(\R)$. 
In particular, when $s = 0$ 
the cubic NLS \eqref{NLS1} on $\R$ is (essentially) scaling-critical in $M^{2, \infty}(\R)$.
We point out that 
a typical function in these critical spaces $\FL^\infty(\M)$
and $M^{2, \infty}(\R)$ is the Dirac delta function
and that 
\eqref{NLS1} on $\R$
and \eqref{NLS2} on $\T$ are known to be ill-posed with the Dirac delta function as initial data;
see \cite{KPV, FO}.
See also  Banica-Vega \cite{BV1, BV2}  for the work on the cubic NLS \eqref{NLS1}
with the Dirac delta function
as initial data.

The Cauchy problems \eqref{NLS1} on $\R$ and \eqref{NLS2} on $\T$
have been studied in the context of the Fourier-Lebesgue spaces
and the modulation spaces.
In particular, local well-posedness in almost critical spaces have been known.
In \cite{GRUN}, Gr\"unrock studied the cubic NLS \eqref{NLS1} on $\R$ in the Fourier-Lebesgue spaces
and proved local well-posedness in $\FL^p(\R)$, $1 < p < \infty$, 
almost reaching the critical case $p = \infty$.
He also proved global well-posedness for $2 \leq p < \frac 5 2$.
We also mention a precursor to this result by Vargas-Vega \cite{VV}, 
establishing well-posedness of \eqref{NLS1} on $\R$
with infinite $L^2$-norm initial data.
In a recent paper, 
S.\,Guo~\cite{Guo} proved local well-posedness of the cubic NLS \eqref{NLS1} on $\R$
in $M^{2, p}(\R)$ for $2 \leq p < \infty$.
In the periodic setting, 
Gr\"unrock-Herr \cite{GH} proved 
 local well-posedness of the renormalized cubic NLS \eqref{NLS2} in $\FL^p(\T)$, $1 < p < \infty$.
See also Christ 
\cite{CH2} for a construction of solutions to \eqref{NLS2}
(without uniqueness) via a power series expansion.
In the same paper, Christ 
also refers to an unpublished work with Erdo\v{g}an,
claiming 
small data global well-posedness in $\FL^p(\T)$, $p < \infty$.
We point out that, as a consequence of the local well-posedness result in \cite{GH}, 
 the non-existence result 
 in \cite{GO} also applies 
 to  the cubic NLS \eqref{NLS1} in 
the Fourier-Lebesgue setting: $\FL^p(\T) \setminus L^2(\T)$, $p>2$.

While there are some global well-posedness results in the context of
the modulation and Fourier-Lebesgue spaces, 
it is very far from matching the local well-posedness regularities.
In the following, we close this gap and prove global well-posedness
in almost critical spaces.
For this purpose, we first introduce  the following modulated Sobolev
space $\MH^{\ta, p}(\R)$ by the norm:
\begin{align}
\|f\|_{\MH^{\ta, p}(\R)}
& = \bigg(\sum_{n \in \Z} \| M_n f\|_{H^\ta}^p \bigg)^\frac{1}{p}\notag\\
& = \bigg(\sum_{n \in \Z} \| \jb{\xi - n}^{\ta} \ft f(\xi)\|_{L^2_\xi}^p \bigg)^\frac{1}{p}, 
\label{mod4}
\end{align}

\noi
where $M_n$ denotes the modulation operator defined by 
\begin{align}
M_nf(x)  = e^{-inx}f(x).
\label{mod4a}
\end{align}

\noi
On the circle, we define $\MH^{\ta, p}(\T)$ in an analogous manner.
When $\ta\geq 0$, we have
$\|f \|_{\MH^{\ta, p}} < \infty$
if and only if $f = 0$.
Hence, we focus on $\ta < 0$ in the following.
In fact, when $\ta < -\frac 12$, it is easy to see that the $\MH^{\ta, p}$-norm 
is equivalent to the $M^{2, p}$-norm.	

\begin{lemma}\label{LEM:equiv}
\textup{(i)} Let $\ta < -\frac{1}{2}$ and $2\leq p \leq \infty$.
Then, we have
\[ \| f \|_{\MH^{\ta, p}} \sim \|f\|_{M^{2, p}}\]

\noi
with the understanding that $M^{2, p}(\T) = \FL^p(\T)$ on the circle.
\smallskip

\noi
\textup{(ii)}
Let $-\frac 12\leq \ta < 0$ and $2\leq q < p \leq \infty$.
Then, we have
\begin{align}
\|f\|_{M^{2,p}} \les \|f\|_{\MH^{\ta, p}}
& \les 
\begin{cases}
\|f\|_{M^{2,q}},\\
\|f\|_{M^{2,p}_s}, 
\end{cases}
\label{equiv2}
\end{align}

\noi
provided that $\frac{1}{q} > \frac{1}{p} + \frac 12 + \ta$
and $s > \frac12 + \ta$.

\end{lemma}

The proof of this lemma is elementary and is presented in Section \ref{SEC:2}.
This equivalence of the norms for $\ta <  - \frac 12$ allows us
to express the relevant modulation norms on $\R$ and the Fourier-Lebesgue norms on $\T$
in terms of the $\l^p$-sum of the modulated Sobolev norms. 
On the one hand,  we introduce the $\MH^{\ta, p}$-norm
for a PDE purpose
and use it for $\ta = -1$.
On the other hand, when $-\frac 12 \leq \ta < 0$, 
it lies between $M^{2, p}$ and $M^{2, q}$ for 
$q$ satisfying  $\frac{1}{q} > \frac{1}{p} + \frac 12 + \ta$.
Hence, it may be of interest to study finer properties
of $\MH^{\ta, p}$.
One may also replace the weight $\jb{\xi - n}^\ta$
by a general weight function $\phi(\xi - n)$
and define the modulated Sobolev space $\MH^{\phi, p}$ adapted to the weight function $\phi$
via the norm:
\begin{align*}
\|f\|_{\MH^{\phi, p}(\R)}
& = \bigg(\sum_{n \in \Z} \| \phi(\xi - n) \ft f(\xi)\|_{L^2_\xi}^p \bigg)^\frac{1}{p}.
\end{align*}

\noi
Arguing as in the proof of Lemma \ref{LEM:equiv}, 
one can easily prove that 
\[ \| f \|_{\MH^{\phi, p}} \sim \|f\|_{M^{2, p}}\]

\noi
for $\phi \in L^2(\R)$ which is bounded away from 0  on $[-\frac 12, \frac 12]$.

\subsection{Main result}

We briefly go over the main result in the work \cite{KVZ} by Killip-Vi\c{s}an-Zhang.
See \cite{KVZ} for more details.
The one-dimensional cubic NLS \eqref{NLS1}
is a completely integrable PDE  and it admits the following Lax pair formulation \cite{ZS, AKNS}:
\[
\frac{d}{dt} L(t;\kk) = \big[P(t,\kk), L(t;\kk)\big], 
\]

\noi
where 
\[
L(t;\kk) 
= \begin{pmatrix}
-\dx+\kk & i u \\ \mp i\cj u & -\dx-\kk 
\end{pmatrix}
\]

\noi
and  $P(t;\kk)$ denotes some operator pencil
whose precise form does not play any role in the following. 
In \cite{KVZ}, the authors studied the following perturbation determinant $\al(\kk; u)$:
\begin{align}
\al(\kk; u) = \Re \sum_{j = 1}^\infty
\frac{(\mp 1)^{j-1}}{j} \tr
\Big\{ \big[(\kk - \dx)^{-\frac{1}{2}} u (\kk + \dx)^{-1} \cj u (\kk - \dx)^{-\frac{1}{2}}\big]^j\Big\}.
\label{X1}
\end{align}

\noi
Here, the operators $(\kk\pm\dx)^{-1}$ and $(\kk\pm\dx)^{-\frac12}$  
are defined as the Fourier multiplier operators.
For an operator $A$ on $L^2(\M)$ with a continuous  integral kernel $K(x,y)$, 
we  define its trace by
\[\tr(A)=\int_{\M} K(x,x) dx.\]

\noi
In particular, if $A$ is a Hilbert-Schmidt operator with an integral kernel $K(x,y)\in L^2(\M^2)$, 
then we have 
\[\tr (A^2) = \iint_{\M^2} K(x,y) K(y,x) dxdy.\]

\noi
We also set
\begin{align*}
\| A \|_{ \If_2}^2 =  \tr(A^* A) = \iint_{\M^2} | K(x,y)|^2 dxdy.
\end{align*}

\noi
Recall from   \cite[Lemma 1.4]{KVZ} that
\begin{align}
|\tr(A_1 \cdots A_k)| \leq \prod_{j = 1}^k \|A_j\|_{\If_2}.
\label{X1x}
\end{align}

In the following, we summarize three important properties of $\al(\kk; u)$.
Here, we only state the real line case.
For the periodic case, the corresponding statements are
basically true with a small change in 
\eqref{X2} for the leading term in the series \eqref{X1};
see Lemma \ref{LEM:main1} below.

\smallskip

\begin{lemma} \label{LEM:KVZ}
The following statements hold:
\begin{itemize}
\item[(i)] 
\cite[Proposition 4.3]{KVZ}:
For a Schwartz class  solution $u$ to \eqref{NLS1}, 
the quantity $\al(\kk; u)$ is conserved,
provided that $\kk>0$ is sufficiently large such that 
\begin{align}
\int_\R  \log\big(4 + \tfrac{\xi^2}{\kk^2}\big)\frac{|\ft u(\xi)|^2}{\sqrt{4 \kk^2 + \xi^2 }} d\xi \le c_0
 \label{X1a}
\end{align}

\noi
for some absolute constant $c_0 > 0$.

\medskip

\item[(ii)] \cite[Lemma 4.2]{KVZ}:
The leading term of the series expansion \eqref{X1} is given by 
\begin{align}
\Re \tr \big\{(\kk - \dx)^{-1} u (\kk + \dx)^{-1} \cj u 
\big\} 
= \int_{\R} \frac{2\kk |\ft u(\xi)|^2}{4 \kk^2 + \xi^2 } d\xi
\label{X2}
\end{align}

\noi
for any  $\kk > 0$ and  $u \in \S(\R)$.

\medskip

\item[(iii)]
\cite[Lemma 4.1]{KVZ}:
We have
\begin{align}
\big\|
(\kk - \dd_x)^{-\frac{1}{2}} u (\kk + \dd_x)^{-\frac 12}\big\|_{\mathfrak{I}_2(\R)}^2
\sim 
\int_{\R} \log\big(4 + \tfrac{\xi^2}{\kk^2}\big)\frac{|\ft u(\xi)|^2}{\sqrt{4 \kk^2 + \xi^2 }} d\xi
\label{X3}
\end{align}

\noi
for any  $\kk > 0$ and  $u \in \S(\R)$.

\end{itemize}

\end{lemma}

In view of \eqref{X1x} and \eqref{X3},
this smallness condition \eqref{X1a} guarantees
term-by-term differentiation of the series \eqref{X1}.
Using the properties  (i) - (iii),
Killip-Vi\c{s}an-Zhang \cite{KVZ}
proved the following global-in-time bound (Theorem 1.3 in \cite{KVZ})
on the $H^s$-norm of smooth solutions to \eqref{NLS1} on $\M = \R$ or $\T$:
\begin{align}
 \| u(t) \|_{H^s} 
\les \|u_0\|_{H^s}\Big(1 + \|u_0\|_{H^s}\Big)^\frac{|s|}{1-|s|}
\label{bound1}
\end{align}

\noi
for $ - \frac 12 < s < 0$.
The main idea of the argument in \cite{KVZ}
is to express the $H^s$-norms
(and in fact the Besov norms)
as a suitable sum
of the right-hand side of \eqref{X2}
as $\kk$ ranges over dyadic numbers $\kk_0 \cdot 2^\NB$
(for some $\kk_0 >0$).
See Lemma 3.2
and the $Z_{\kk_0}$-norm in the proof of Theorem 4.5
in \cite{KVZ}.
The property (iii) above is then to control the error
terms (i.e.~$j \geq 2$) in \eqref{X1},
which imposes the restriction $s>-\frac 12$.
Note that the restriction 
$s>-\frac 12$ is necessary in view of the norm inflation 
at the critical regularity $s = - \frac 12$ \cite{Kishimoto, OW, O1}.

We point out that the global-in-time bound \eqref{bound1} also holds
for smooth solutions $u$ to the renormalized cubic NLS \eqref{NLS2} on $\T$
since we can convert smooth solutions
to the cubic NLS \eqref{NLS1}  
and to the renormalized cubic NLS \eqref{NLS2}
by the following invertible gauge transform:
\begin{align}
\mathcal{J}(u)(t) : = e^{\mp 4 i t \fint |u(t)|^2 dx } u(t), 
\label{gauge}
\end{align}

\noi
while the gauge transform $\mathcal{J}$ preserves the $H^s$-norm.
As mentioned above, 
 uniqueness of solutions to
 the cubic NLS \eqref{NLS1} on $\R$
and the renormalized cubic NLS \eqref{NLS2} on $\T$
in negative Sobolev spaces 
remains as a very challenging open question.
Hence, 
while the global-in-time bound \eqref{bound1} may be used to prove
global existence of solutions (without uniqueness)
in negative Sobolev spaces, 
it does not provide global well-posedness at this point.

In the following, we establish global-in-time bound
on the $M^{2, p}$-norm
of  smooth solutions 
to the cubic NLS \eqref{NLS1} on $\R$
and the $\FL^p$-norm of
smooth solutions 
to  the renormalized cubic NLS \eqref{NLS2} on $\T$.
Then, the local well-posedness
in these spaces \cite{Guo, GH}
yields the following global well-posedness result.

\begin{theorem}\label{THM:1}
Let $2 \leq p < \infty$. 

\begin{itemize}
\item[\textup{(i)}] 
There exists $C= C(p) >0$ such that 
\begin{align}
\|u(t)\|_{M^{2, p}(\R)}\le C (1+\|u(0)\|_{M^{2, p}(\R)})^{\frac p2 - 1}\|u(0)\|_{M^{2, p}(\R)}
\label{bd1}
\end{align}

\noi
for any Schwartz class  solution $u$ to \eqref{NLS1} on $\R$
and any $t \in \R$.
In particular, 
the cubic NLS \eqref{NLS1} on $\R$ is globally well-posed
in $M^{2, p}(\R)$.

\smallskip

\item[\textup{(ii)}] 
There exists $C= C(p) >0$ such that 
\begin{align}
\|u(t)\|_{\FL^p(\T)}\le C \big(1+\|u(0)\|_{\FL^p(\T)}\big)^{\frac p2 - 1} \|u(0)\|_{\FL^p(\T)}
\label{bd2}
\end{align}

\noi
for any smooth solution $u$ to \eqref{NLS1} on $\T$
and any $t \in \R$.
In particular, 
the renormalized cubic NLS \eqref{NLS2} on $\T$ is globally well-posed
in $\FL^p(\T)$.

\end{itemize}

\end{theorem}

Theorem \ref{THM:1} establishes global well-posedness
of the cubic NLS \eqref{NLS1} on $\R$ and the renormalized cubic NLS \eqref{NLS2}on $\T$
in almost critical spaces, 
improving significantly the known global well-posedness results \cite{VV, GRUN, CH2}.
Moreover, the range of $p< \infty$ of Theorem~\ref{THM:1} 
is sharp in view of the ill-posedness results
with the Dirac delta as initial data which lies in $M^{2, \infty}(\R)$ and $\FL^\infty(\T)$.\footnote{In \cite{FO}, the ill-posedness result with the Dirac initial data on $\T$ was shown in a topology weaker than $\FL^\infty(\T)$.
We also point out that  $\FL^\infty(\T)$ does not admit smooth approximations
and hence  an a priori bound on the $\FL^{\infty}(\T)$-norm
for  smooth solutions would not yield the same bound for rough solutions.}
See also Remark \ref{REM:G} below.
In view of the local well-posedness \cite{Guo, GH}, 
it suffices to  establish a priori global-in-time bounds
\eqref{bd1} on $\R$ and \eqref{bd2} on $\T$, 
controlling 
the $M^{2, p}$-norms\footnote{In view of \eqref{mod5} on the circle, 
we  may  use $M^{2, p}(\T)$ for $\FL^p(\T)$ in the following.} of smooth solutions 
 to the cubic NLS \eqref{NLS1} on $\M = \R$ or $\T$.
In the periodic setting, 
we then obtain the same global-in-time bound \eqref{bd2}
for smooth solutions to the renormalized cubic NLS \eqref{NLS2}
via the gauge transform \eqref{gauge}, 
yielding global well-posedness for the renormalized cubic NLS \eqref{NLS2} on $\T$.
The main idea for proving Theorem \ref{THM:1}
is to use 
the equivalence 
of the norms for the modulation spaces $M^{2, p}$
and the modulated Sobolev spaces $\MH^{-1, p}$
(Lemma \ref{LEM:equiv}).
We then apply the result in \cite{KVZ}
to control the growth of the $\MH^{-1, p}$-norm.
Here, both the scaling and Galilean symmetries
play an important role. 
It follows from \eqref{mod4}, \eqref{X1}, and \eqref{X2}
that 
(the square of) the $\MH^{-1, p}$-norm of a solution $u$
is given by  the $\l^\frac{p}{2}$-sum of the leading terms 
for $\al(\frac 12; \mathcal{G}_n(u))$, $n \in \Z$, 
where $\mathcal{G}_n$ is the Galilean transform defined  in \eqref{Galilei}.
On the one hand, 
the Galilean symmetry and Lemma \ref{LEM:KVZ}
imply the conservation of  $\al(\frac 12; \mathcal{G}_n(u))$
in the small data case.
On the other hand, the scaling symmetry with the subcriticality 
of the underlying space $M^{2, p}$, $p< \infty$, allows us to reduce the situation to the small data case
and handle the error terms in the series \eqref{X1}. 
Compare this with \cite{KVZ}, 
where the main idea in this step is to express the $H^s$-norm
as a suitable sum of the leading terms of $\al(\kk; u)$
as $\kk$ ranges over dyadic numbers $\kk_0 \cdot 2^\NB$
(for some $\kk_0 \gg 1$).
Here, taking $\kk_0\gg 1$ essentially has an effect of scaling, 
reducing to the small data case (in the subcritical regularity $s > -\frac 12$).
Lastly, we  note that, 
as in \cite{KVZ}, 
we can control the error terms only in the subcritical range,
i.e.~$p < \infty$ in our setting.

\begin{remark}\label{REM:mKdV bound}
\rm

(i) 
In Appendix \ref{SEC:mKdV}, 
we consider the following complex-valued modified KdV equation on $\M = \R$ or $\T$:
\begin{align*}
 \dt u = - \dx^3 u \pm 6 |u|^2 \dx u
\end{align*}

\noi
and briefly discuss how to derive the same global-in-time bounds \eqref{bd1}
and \eqref{bd2} for Schwartz/smooth solutions to the mKdV.
See Theorem \ref{THM:2}.

\smallskip

\noi
(ii) The global-in-time bounds \eqref{bd1} and \eqref{bd2}
can be extended to 
the modulation spaces  $M^{2, p}_s(\R)$ and 
the Fourier-Lebesgue spaces $\FL^{s, p}(\T)$
of higher regularities.
See Appendix \ref{SEC:B}.

\end{remark}

\begin{remark} \label{REM:G} \rm
In \cite{Guo}, S.\,Guo  proved local well-posedness of the cubic NLS \eqref{NLS1}
on $\R$ in a space whose norm  is logarithmically stronger than the critical $M^{2, \infty}(\R)$-norm.
The space is characterized by an Orlicz norm and contains functions whose
Fourier transforms decay only logarithmically, i.e.~not belonging to $M^{2, p}(\R)$ for any finite $p< \infty.$
It seems of interest to study the global-in-time behavior of solutions 
in this logarithmically subcritical space.
We also remark that, in   \cite{KVZ}, Killip-Vi\c{s}an-Zhang
also established global-in-time bounds for solutions to the cubic NLS \eqref{NLS1}
on $\R$ and $\T$ in negative Besov-type spaces which are logarithmically stronger
than the critical $H^{-\frac{1}{2}}$, where the norm inflation \eqref{NI} is known.

\end{remark}

\begin{remark} \rm
In \cite{CO}, the first author (with Colliander) studied the renormalized cubic NLS \eqref{NLS2}
on $\T$
with random initial data of the form:
\begin{align}
u_0^\al (x) = \frac{1}{\sqrt {2\pi}} \sum_{n \in \Z} \frac{g_n}{\jb{n}^\al} e^{inx},
\label{random1}
\end{align}

\noi
where $\{g_n \}_{n \in \Z}$ is a sequence of independent standard complex-valued
Gaussian random variables.
It is easy to see that such $u_0^\al$ almost surely  belongs
to $H^{\al - \frac 12 - \eps}(\T)\setminus
H^{\al - \frac 12}(\T)$.
When $\al = 0$, this corresponds to the white noise on $\T$
and is of significant importance to study  \eqref{NLS2}
with the white noise initial data.
It is also easy to see that $u_0^\al$ in \eqref{random1}
almost surely belongs to $\FL^p(\T)$ for $p > \al^{-1}$.
Therefore, Theorem \ref{THM:1} (ii) 
yields (a deterministic proof of) almost sure global well-posedness of \eqref{NLS2}
with almost white noise initial data $u_0^\al$, $\al > 0$.
The $\al = 0$ case remains as an important open problem.

\end{remark}

\section{Equivalence of the norms}
\label{SEC:2}

In this section, we present a proof of Lemma \ref{LEM:equiv}.
In the following, we only consider the real line case since the proof for 
the periodic case follows in a similar manner.
Obviously,   we have
\[
\|f\|_{M^{2,p}} \les \|f\|_{\MH^{\ta, p}}
\]
since $\psi_n(\xi)  \les \jb{\xi-n}^\ta$,
where $\psi_n$ is as in \eqref{psi1}.

Let $I_k = [k - \frac 12, k + \frac12)$, $k \in \Z$.
By writing
\begin{align*}
 \|f\|_{\MH^{\ta, p}} 
 & = \Big\| \|\jb{\xi-n}^{\ta} \ft f (\xi) \|_{L^2}\Big\|_{\l^p_n}\\
 & \sim \bigg\|  \sum_{k\in\Z} \jb{k-n}^{2\ta} \int_{I_k}  | \ft f (\xi)|^2  d \xi \bigg\|_{\l^\frac{p}{2}_n}^\frac{1}{2}, 
\end{align*}

\noi
it follows from Young's inequality with $p \geq 2$ that 
\begin{align*}
\|f\|_{\MH^{\ta, p}} 
& \les \bigg( \sum_{n\in \Z} \jb{n}^{2\ta} \bigg)^\frac{1}{2} 
\bigg\| \int_{I_n}  | \ft f (\xi)|^2  d \xi \bigg\|_{\l^\frac{p}{2}_n}^\frac{1}{2}\\
& \les \|f\|_{M^{2, p}}, 
\end{align*}

\noi
provided that $\ta < - \frac12$.
This proves (i).

Next, we consider the case $-\frac 12 \leq \ta < 0$.
In this case, we need to lose either integrability or differentiability.
By Young's inequality with $\frac 2p + 1 = \frac 1r + \frac 2q$, we have
\begin{align}
\|f\|_{\MH^{\ta, p}} 
& \les \bigg( \sum_{n\in \Z} \jb{n}^{2\ta r}\bigg)^{\frac{1}{2r}}
\bigg\| \int_{I_n}  | \ft f (\xi)|^2  d \xi \bigg\|_{\l^\frac{q}{2}_n}^\frac{1}{2} \notag \\
& \les \|f\|_{M^{2,q}},
\label{2.1}
\end{align}

\noi
provided that $2\ta r < -1$, namely, 
$\frac{1}{q} > \frac{1}{p} + \frac 12 + \ta$.
The second bound  in \eqref{equiv2} follows from 
applying H\"older's inequality to the right-hand side of \eqref{2.1}.

\section{Control on the modulation and Fourier-Lebesgue norms}

In this section, we present the proof of Theorem \ref{THM:1}.
In view of the local well-posedness results
\cite{GRUN, GH}, it suffices to establish global-in-time bounds
for smooth solutions to \eqref{NLS1} on $\M = \R$ or $\T$.

\subsection{On the real line}
We first consider the real line case.
Our main goal is to prove the following global-in-time bound.

\begin{proposition}\label{PROP:main1}
Let $2\le p<\infty$.
Then, there exists $C= C(p) >0$ such that 
\begin{align*}
\|u(t)\|_{M^{2, p}(\R)}\le C (1+\|u(0)\|_{M^{2, p}(\R)})^{\frac p2 - 1}\|u(0)\|_{M^{2, p}(\R)}
\end{align*}

\noi
for any Schwartz class  solution $u$ to \eqref{NLS1} on $\R$
and any $t \in \R$.

\end{proposition}

\begin{proof}
Let us first consider the small data case.
The general case follows from the small data case and the scaling 
property of the $M^{2, p}$-norm.
Fix $2\leq p < \infty$.
Let $u$ be a global-in-time Schwartz class solution  to \eqref{NLS1}, satisfying
\begin{align}
\| u(0)\|_{M^{2, p}} \leq \eps \ll 1
\label{scaling3}
\end{align}

\noi
for some small $\eps >0$ (to be chosen later).
Given $n \in \Z$, 
define $\{u_n\}_{n\in \Z}$ by 
\begin{align}
u_n(x, t) = \mathcal{G}_n(u)(x, t) =  e^{-inx} e^{in^2 t } u ( x - 2n t, t),
\label{Y0b}
\end{align}

\noi
where $\mathcal{G}_n$ is as in \eqref{Galilei}.
Note that we have
\begin{align}
|\ft u_n(\xi, t)|  = |\ft u(\xi + n, t)|
\label{Y0a}
\end{align}

\noi
for any $n \in \Z$ and $\xi, t\in \R$.
In view of the Galilean symmetry, 
$u_n$ is the solution to \eqref{NLS1} with $u_n|_{t = 0} = M_n u(0)$,
where $M_n$ is as in \eqref{mod4a}.

In the following, we fix $\kk  = \frac 12$
and set $\al (u) = \al(\frac 12 ; u)$.
From \eqref{X1}, \eqref{X1x}, \eqref{X2},  and \eqref{X3} with \eqref{Y0a}, 
we have
\begin{align}
\bigg| \al( u_n(t)) 
- \int_\R \frac{ |\ft u_n(\xi, t)|^2}{1 + \xi^2 } d\xi\bigg|
& \leq  \sum_{j = 2}^\infty
\frac{1}{j}\bigg\| 
(\tfrac 12 - \dx)^{-\frac{1}{2}} \, u_n(t)\,  (\tfrac 12 + \dx)^{-\frac 12}\bigg\|_{\mathfrak{I}_2(\R)}^{2j}\notag\\
& \les \sum_{j = 2}^\infty
\bigg(\int_\R \frac{|\ft u_n(\xi, t)|^2}{(1 + \xi^2 )^{\frac{1}{2}-\dl}} d\xi\bigg)^j \notag\\
& \les \sum_{j = 2}^\infty
\bigg(\int_\R \frac{|\ft u(\xi, t)|^2}{(1 + (\xi - n)^2 )^{\frac{1}{2}-\dl}} d\xi\bigg)^j
\label{Y1}
\end{align}

\noi
for any  $\dl > 0$.
By H\"older's inequality, 
we can choose sufficiently small  $\dl = \dl(p) > 0$ such that 
\begin{align}
\int_\R \frac{|\ft u(\xi, 0)|^2}{(1 + (\xi- n)^2 )^{\frac{1}{2}-\dl}} d\xi
& \sim \sum_{k \in \Z} 
\frac{1}{(1 + (k - n)^2 )^{\frac{1}{2}-\dl}} \| \ft{u(0)} \|_{L^2_\xi(I_k)}^2 \notag\\
& \les  \| u(0) \|_{M^{2, p}}^2
\label{Y1a}
\end{align}

\noi
uniformly in $n \in \Z$, 
where $I_k = [k - \frac 12, k + \frac12)$, $k \in \Z$, as above.
Then,
in view of  \eqref{scaling3} and \eqref{Y1a}, 
we can choose 
 $\eps > 0$ sufficiently small  such that the series
on the right-hand side of \eqref{Y1} is convergent at time $t = 0$.
Then, by continuity  in time, there exists a small time interval $I$ around $t = 0$
such that the  series
on the right-hand side of \eqref{Y1} is convergent
uniformly for any $t \in I$.
Moreover, by choosing $\eps > 0$ sufficiently small, 
we may assume that \eqref{X1a} is satisfied for all $t \in I$.

As a consequence,  we have
\begin{align*}
\bigg| \al(u_n(t)) 
- \int_\R \frac{ |\ft u_n(\xi, t))|^2}{1 + \xi^2 } d\xi\bigg|
& \les 
\bigg(\int_\R \frac{|\ft u(\xi, t)|^2}{(1 + (\xi-n)^2 )^{\frac{1}{2}-\dl}} d\xi\bigg)^2
\end{align*}

\noi
for any $t \in I$ and $n \in \Z$.
Now, compute the $\l^\frac{p}{2}_n$-norm of both sides.
Choose $\dl = \dl(p) > 0$ sufficiently small
such that $ (\frac12 - \dl)\cdot \frac{p}{p-1} > \frac{1}{2}$.
Then, by Young's inequality, 
we have
\begin{align}
 \bigg\| \al( u_n(t)) 
& - \int_\R \frac{ |\ft u_n(\xi, t)|^2}{1 + \xi^2 } d\xi\bigg\|_{\l^\frac{p}{2}_n}\notag\\
& \les 
\bigg\|\int_\R \frac{|\ft u(\xi, t)|^2}{(1 + (\xi-n)^2 )^{\frac{1}{2}-\dl}} d\xi\bigg\|_{\l^p_n}^2\notag\\
& \sim
\bigg\|\sum_{k \in \Z} 
 \frac{1}{(1 + (k-n)^2 )^{\frac{1}{2}-\dl}}
\| \ft u(\xi, t)\|_{L^2_\xi(I_k)}^2\bigg\|_{\l^p_n}^2\notag\\
& \les
 \Big\| \| \ft u(\xi, t)\|_{L^2_\xi(I_n)}^2\Big\|_{\l^\frac{p}{2}_n}^2\notag\\
& \sim  \| u(t)\|_{M^{2, p}}^4
\label{Y3}
\end{align}

\noi
for any $t \in I$.
Therefore, from Lemma \ref{LEM:equiv}, \eqref{Y3}, and 
the conservation of $\al(u_n)$, $n \in \Z$, 
\begin{align*}
\| u(t) \|_{M^{2, p}}^2
& \sim 
\|  u (t) \|_{\MH^{-1, p}}^2
 \le\| \al( u_n(t)) \|_{\l^\frac{p}{2}_n}
+   \| u(t)\|_{M^{2, p}}^4 \notag\\ 
& \les
\| u(0) \|_{\MH^{-1, p}}^2
+   \| u(0)\|_{M^{2, p}}^4
+   \| u(t)\|_{M^{2, p}}^4\notag\\
& \les
\|u(0) \|_{M^{2, p}}^2
+  \| u(0)\|_{M^{2, p}}^4
+  \| u(t)\|_{M^{2, p}}^4
\end{align*}

\noi
for all $t \in I$.
Namely, we have 
\begin{align*}
\| u(t) \|_{M^{2, p}}^2
& \le
C_0\|u(0) \|_{M^{2, p}}^2
+ C_0 \| u(0)\|_{M^{2, p}}^4
+  C_0\| u(t)\|_{M^{2, p}}^4
\end{align*}

\noi
for all $t \in I$.
By  choosing $\eps > 0$ sufficiently small, 
we can apply a continuity argument and conclude that
\begin{align*}
\| u(t) \|_{M^{2, p}}
& \les\|u(0) \|_{M^{2, p}}
\end{align*}

\noi
for all $t \in \R$.
This proves Proposition \ref{PROP:main1} for the small data case.

Next, we consider the general case.
Given 
a  global-in-time Schwartz class solution $u$ to \eqref{NLS1}, 
let $u_\ld$ be as in \eqref{scaling}.
Then 
in view of \eqref{scaling2}, we can choose 
sufficiently large $\ld \gg 1$
such that 
\begin{align}
\| u_\ld (0)  \|_{ M^{2, p}} \le C\ld^{ - \frac 1p} \| u(0)\|_{ M^{2, p}}
\le \eps \ll1
\label{Y4}
\end{align}

\noi
as in \eqref{scaling3}.
In particular, we may choose 
\begin{align}
\ld \sim  (1+ \| u(0)\|_{ M^{2, p}})^p.
\label{Y4a}
\end{align}

\noi
Hence, by the small data case presented above, we obtain
\begin{align}
\| u_\ld(t) \|_{M^{2, p}}
 \les\|u_\ld(0) \|_{M^{2, p}}
\label{Y5}
\end{align}

\noi
for all $t \in \R$.
Finally, recall that 
\begin{align}
\| u(t) \|_{M^{2, p}}
 \les \ld^{\frac 12} \|u_\ld(\ld^2 t) \|_{M^{2, p}}.
\label{Y6}
\end{align}

\noi
By putting \eqref{Y4}, \eqref{Y5}, and \eqref{Y6}
with \eqref{Y4a}, we conclude that 
\begin{align*}
\| u(t) \|_{M^{2, p}}
 \les 
 (1+ \| u(0)\|_{ M^{2, p}})^{\frac{p}{2} - 1}
 \|u(0) \|_{M^{2, p}}
\end{align*}

\noi
for all $t \in \R$.
This completes the proof of Proposition \ref{PROP:main1}
and hence the proof of Theorem~\ref{THM:1}\,(i).
\end{proof}

\subsection{On the circle}

In the remaining part of this paper,  we discuss the proof of Theorem~\ref{THM:1}
in  the periodic case.
While the essential part of the argument remains the same, 
we need to pay attention to the scaling argument since it modifies the spatial domain.
Given $\ld >0$, let $\T_\ld = \R/(2\pi \ld \Z)$
and we use the following convention:
\begin{align}
 \ft f(\xi) = \frac{1}{\sqrt{2\pi}} \int_0^{2\pi\ld} f(x) e^{-ix \xi} dx
\qquad\text{and}\qquad
 f(x) =\frac{1}{\sqrt{2\pi}\ld} \sum_{\xi \in  \Z_\ld} \ft f(\xi) e^{ix \xi}
\label{Z0}
\end{align}

\noi
for functions on the dilated torus  $\T_\ld$, where $\Z_\ld = \ld^{-1} \Z$.
In this setting, 
Plancherel's identity is expressed as 
\begin{align}
 \|f \|_{L^2(\T_\ld)} = \| \ft f(\xi) \|_{L^{2}_\xi(\Z_\ld, (d\xi)_\ld)},
\label{Z0a}
 \end{align}

\noi
where $(d \xi)_\ld$ is the normalized counting measure on $\Z_\ld$:
\[ \int_{\Z_\ld} f(\xi) (d\xi)_\ld = \frac{1}{\ld} \sum_{\xi \in \Z_\ld}  f(\xi).\]

\noi
Hence, we define  the Fourier-Lebesgue space $\FL^p(\T_\ld)$ by the norm:
\begin{align*}
\|f \|_{\F L^{p}(\T_\ld)} = \| \ft f(\xi) \|_{L^{p}_\xi(\Z_\ld, (d\xi)_\ld)}.
\end{align*}

\noi
For simplicity of the notation, 
we set
$L^{p}_\xi(\Z_\ld) = L^{p}_\xi(\Z_\ld, (d\xi)_\ld)$.
Under this convention, we have the following scaling property:
\begin{align*}
\|f_\ld \|_{\F L^{p}(\T_\ld)} =  \ld^{-\frac{1}{p}}\|f \|_{\F L^{p}(\T)} , 
\end{align*}

\noi
where  
\begin{align}
f_\ld(x) = \ld^{-1} f(\ld^{-1}x).
\label{Z1a}
\end{align}
In particular, note that the $\FL^\infty$-norm is invariant under the scaling symmetry.
We also  record the following identity:
\begin{align}
\ft{fg}(\xi) = \frac{1}{\sqrt{2\pi}} \int_{ \Z_\ld} \ft f(\eta) \ft g(\xi - \eta)(d\eta)_\ld
\label{Z1b}
\end{align}

\noi
for $\xi \in \Z_\ld$.

In the following, we restrict our attention to $\ld \geq 1$.
On the standard torus, we have the equivalence of the modulation spaces
and the Fourier-Lebesgue spaces (see \eqref{mod5}).
On a dilated torus $\T_\ld$, $\ld \gg 1$, however, these two spaces do not coincide.
As in \eqref{mod}, we define the modulation spaces on $\T_\ld$ by the norm:
\begin{equation*} 
\|f\|_{M^{r, p} (\T_\ld)} 
= \big\| \|\psi_n(D) f \|_{L_x^r(\T_\ld)} \big\|_{\l^p_n(\Z)}.
\end{equation*}

\noi
Note that $\psi_n(D)$ is basically a ``smooth''
projection onto the frequencies $J_n = [n-1, n+1]\cap \Z_\ld$.
On the one hand, when $\ld = 1$,  
we have only $O(1)$ many  frequencies in $J_n$, giving rise to the equivalence \eqref{mod5}.
On the other hand, when $\ld \gg 1$, 
there are $O(\ld)$ 
 many  frequencies in $J_n$
 and hence the $M^{r, p}(\T_\ld)$- and the $\FL^p(\T_\ld)$-norms are 
 not equivalent uniformly in period $\ld\geq 1$.
As we see below, for our purpose, it is more convenient to work on the modulation space $M^{2, p}(\T_\ld)$
on a dilated torus $\T_\ld$.
In view of Plancherel's identity \eqref{Z0a}, we have
\begin{align*} 
\|f\|_{M^{2, p} (\T_\ld)} 
& = \big\| \|\psi_n(\xi) \ft f(\xi) \|_{L^2_\xi(\Z_\ld)} \big\|_{\l^p_n(\Z)}\notag\\
& = \bigg\| \bigg(\int_{ \Z_\ld} |\psi_n(\xi) \ft f(\xi)|^2 (d\xi)_\ld\bigg)^\frac{1}{2} \bigg\|_{\l^p_n(\Z)}.
\end{align*}

Given a function $f$ on $\T$, let $f_\ld$ be as in \eqref{Z1a}.
Then, 
a straightforward computation shows the following scaling relation:
\begin{align}
\| f\|_{\FL^p(\T)} & \les \ld^{\frac 12}\|f_\ld\|_{M^{2, p}(\T_\ld)}, 
\label{Z3}\\
\|f_\ld\|_{M^{2, p}(\T_\ld)} & \les \ld^{-\frac{1}{p}}\| f\|_{\FL^p(\T)}
\label{Z4}
\end{align}

\noi
for $2\le p \le \infty$.

Lastly, we define the modulated Sobolev 
space $\MH^{\ta, p}(\T_\ld)$ on a dilated torus $\T_\ld$ by the norm:
\begin{align*}
\|f\|_{\MH^{\ta, p}(\T_\ld)}
& = \big\| \| M_n f\|_{H^\ta(\T_\ld)} \big\|_{\l^p_n(\Z)}\notag\\
& = \bigg( \sum_{n \in \Z} \| \jb{\xi - n}^{\ta} \ft f(\xi)\|_{L^2_\xi(\Z_\ld)}^p \bigg)^\frac{1}{p}, 
\end{align*}

\noi
where $M_n$ is as in \eqref{mod4a}.
Note that the outer summation ranges over $n \in \Z$ (and not over $\Z_\ld$).
Then, arguing as in the proof of Lemma \ref{LEM:equiv}, 
we obtain the following equivalence when $\ta < -\frac 12$ and $2 \leq p \leq \infty$:
\begin{align}
 \| f \|_{\MH^{\ta, p}(\T_\ld)} \sim \|f\|_{M^{2, p}(\T_\ld)}.
 \label{Z6}
\end{align}

\noi
As in the real line case, the equivalence \eqref{Z6} with $\ta = -1$ plays an important role
in our analysis.

We first state basic properties of the perturbation determinant $\al(\kk; u)$ in \eqref{X1}
on a dilated torus $\T_\ld$.
The following three lemmas are just restatements of Lemma \ref{LEM:KVZ}
on a dilated torus.

\begin{lemma}\label{LEM:error}
Let $\kk > \kk_0$ for some $\kk_0 >0$ and $\ld \geq 1$.
Then, we have
\begin{align}
\big\|(\kk-\dx)^{-\frac 12} u (\kk+\dx)^{-\frac 12} \big\|^2_{\If_2(\T_\ld)} 
\sim \int_{ \Z_\ld} \log\bigl(4 + \tfrac{\xi^2}{\kk^2}\big)\frac{|\ft u(\xi)|^2}{\sqrt{4\kk^2 + \xi^2}}(d\xi)_\ld,
\label{error1}\\
\big\|(\kk+\dx)^{-\frac 12} \cj u (\kk-\dx)^{-\frac 12} \big\|^2_{\If_2(\T_\ld)} 
\sim \int_{ \Z_\ld} \log\bigl(4 + \tfrac{\xi^2}{\kk^2}\big)\frac{|\ft u(\xi)|^2}{\sqrt{4\kk^2 + \xi^2}}
(d\xi)_\ld,
\label{error2}
\end{align}

\noi
for any smooth function $u$ on $\T_\ld$,
where the implicit constant depends on $\kk_0>0$.

\end{lemma}

\begin{proof}
We only consider the first estimate \eqref{error1}
as the second estimate \eqref{error2} follows from a similar computation.
Let $A = (\kk-\dx)^{-\frac 12} u (\kk+\dx)^{-\frac 12}$.
Then, with \eqref{Z0} and \eqref{Z1b},  we have
\begin{align*}
A f (x) & = (\kk-\dx)^{-\frac 12} u (\kk+\dx)^{-\frac 12} f\\
& = \frac{1}{2\pi \ld^2}
\sum_{\xi\in \Z_\ld} e^{i\xi x} (\kk -i\xi)^{-\frac12} 
\sum_{\eta \in \Z_\ld} \ft u (\xi-\eta) (\kk + i \eta)^{-\frac12} \ft f (\eta) \\
& = \frac{1}{2\pi \ld^2} 
\int_{\T_\ld} \bigg( \sum_{\xi, \eta\in \Z_\ld} e^{i\xi x} e^{-i\eta y} 
\frac{\ft u (\xi-\eta)}{\sqrt{2\pi (\kk -i\xi)  (\kk + i \eta)}} \bigg) f (y) dy
\end{align*}

\noi
 for $f\in L^2(\T_\ld)$.
Thus,  the integral kernel $K(x, y)$ of $A$ is given by 
\[
K(x,y) = 
 \frac{1}{2\pi \ld^2} 
 \sum_{\xi, \eta\in \Z_\ld} e^{i\xi x} e^{-i\eta y} 
\frac{\ft u (\xi-\eta)}{\sqrt{ 2\pi (\kk -i\xi)  (\kk + i \eta)}}.
\]

\noi
Hence, from the definition of the $\If_2$-norm and Plancherel's identity \eqref{Z0a}, we have
\begin{align}
\|A\|_{\If_2}^2 & = \iint_{\T_\ld^2} |K(x,y)|^2  dxdy\notag\\
& = \frac{1}{2\pi\ld^2}  \sum_{\xi,\eta \in \Z_\ld} \frac{|\ft u (\xi - \eta)|^2}{\sqrt{(\kk^2 + \xi^2)(\kk^2 + \eta^2)}}.
\label{error3}
\end{align}

\noi
On the other hand, with $\xi = \frac n \ld$ and $\eta = \frac{m}{\ld}$, $n, m \in \Z$, we have
\begin{align}
\frac{1}{\ld^2} \sum_{\eta  \in \Z_\ld}&  
 \frac1{\sqrt{(\kk^2 + (\xi+\eta)^2)( \kk^2 +  \eta^2)}} \notag\\
& =  \sum_{m  \in \Z}  
 \frac1{\sqrt{((\ld\kk)^2 + (n+m)^2)( (\ld\kk)^2 +  m^2)}} \notag \\
 \intertext{By separately estimating the contributions 
 from (i) $|n + m| \ll |n|$, (ii)  $|n+ m|\sim |n|$, and (iii) $|n+m|\gg |n|$,}
&\sim \bigg(\log\big(1 + \tfrac{n^2}{(\ld\kk)^2}\big)
 \frac1{\sqrt{(\ld\kk)^2 +  n^2}}
 + \frac{|n|}{(\ld\kk)^2 + n^2}\bigg)
\notag  \\
&\sim \frac{1}{\ld}\log\big(4 + \tfrac{\xi^2}{\kk^2}\big)
 \frac1{\sqrt{4\kk^2 +  \xi^2}}.
\label{error4}
\end{align}

\noi
Therefore, the first estimate \eqref{error1} follows from \eqref{error3} and \eqref{error4}.
\end{proof}

Next, we estimate the leading term in \eqref{X1}.

\begin{lemma}\label{LEM:main1}
Let $\kk > 0$ and $\ld \geq 1$.
Then, we have
\begin{equation}\label{A1}
\Re \tr\big\{(\kk-\dx)^{-1} u (\kk+\dx)^{-1} \cj u \big\}
 =\frac{1 + e^{-2\pi\ld\kk}} {1 - e^{-2\pi\ld\kk}} 
 \cdot \int_{ \Z_\ld} \frac{2\kk |\ft u (\xi)|^2 }{4\kk^2 + \xi^2}(d\xi)_\ld
\end{equation}

\noi
for any smooth function $u$ on $\T_\ld$.

\end{lemma}

\begin{proof}
Recall from \cite{KVZ} that 
for $\kk > 0$, the operators $(\kk - \dx)^{-1}$ and $(\kk+\dx)^{-1}$, 
 admit the convolution  kernels
$k^\kk_-(x) =  \ind_{(-\infty, 0]}(x) \cdot e^{\kk x}$ and $k^\kk_+(x) = k^\kk_-(-x)$.
Define 
$K^\kk_\mp(x)$ to be the periodization of $k^\kk_\mp(x)$ with period $2\pi \ld$:
\[ K^\kk_\mp(x) = \sum_{n \in \Z} k^\kk_\mp(x - 2\pi \ld n).\]

\noi
Then, the Fourier coefficients are given by
\begin{align*}
\ft{K^\kk_\mp}(\xi) 
& = \frac{1}{\sqrt{2\pi}} \int_0^{2\pi\ld} K^\kk_\mp(x) e^{-ix\xi} dx
= \frac{1}{\sqrt{2\pi}} \int_\R k^\kk_\mp (x) e^{-ix\xi} dx\notag \\
& = \frac{1}{\kk \mp i \xi}
\end{align*}

\noi
for $\xi \in \Z_\ld$.
Namely, 
$K^\kk_\mp(x)$ represents the convolution kernels for 
 $(\kk - \dx)^{-1}$ and $(\kk+\dx)^{-1}$ on $\T_\ld$, respectively.
Also, note that
$K^\kk_+(-x) = K^\kk_-(x)$
 and 
 \begin{align*}
K^\kk_-(x) = \frac{e^{\kk (x - 2\pi \ld  \ceil{\frac{x}{2\pi\ld}})}}{1 - e^{-2\pi\ld\kk}},
 \end{align*}

\noi
where $\ceil{\cdot}$ denotes the ceiling function, i.e.~$\ceil{x}$
denotes the smallest integer $n$ with $n \geq x$.
In particular, we have
 \begin{align*}
\big(K^\kk_-\big)^2(x) 
& = \frac{1 - e^{-4\pi\ld\kk}} {(1 - e^{-2\pi\ld\kk})^2}
\cdot \frac{e^{2\kk (x - 2\pi \ld  \ceil{\frac{x}{2\pi\ld}})}}{1 - e^{-4\pi\ld\kk}}\\
& = \frac{1 + e^{-2\pi\ld\kk}} {1 - e^{-2\pi\ld\kk}}
\cdot 
K^{2\kk}_-\big(x) .
 \end{align*}

\noi
 From these observations and Parseval's identity, 
we have
 \begin{align*}
\text{LHS of \eqref{A1}}  
&= \Re \iint_{\T_\ld^2} K_-(x - y) u (y) K_+(y - x) \cj{u(x)} dx dy\\
&= \Re \iint_{\T_\ld^2} \big(K_-(x - y)\big)^2  u (y) \cj{u(x)} dx dy\\
  &= 
  \frac{1 + e^{-2\pi\ld\kk}} {1 - e^{-2\pi\ld\kk}}\cdot
  \Re \int_{ \Z_\ld} \frac{|\ft u (\xi)|^2}{2\kk - i \xi}(d\xi)_\ld\\
&= \frac{1 + e^{-2\pi\ld\kk}} {1 - e^{-2\pi\ld\kk}}\cdot
\int_{ \Z_\ld} \frac{2\kk|\ft u(\xi)|^2}{4\kk^2 + \xi^2}(d\xi)_\ld.
\end{align*}

\noi
This proves the identity \eqref{A1}.
\end{proof}

The conservation of $\al(\kk;u)$ follows as in the real line case.  

\begin{lemma} \label{LEM:conserved2}
Let $\kk > \kk_0$ for some $\kk_0 >0$ and $\ld \geq 1$.
 For a smooth solution $u$ to \eqref{NLS1} on a dilated torus, 
the quantity $\al(\kk; u)$ defined in \eqref{X1} is conserved,
provided that 
\begin{align}
\int_{ \Z_\ld} \log\big(4 + \tfrac{\xi^2}{\kk^2}\big)\frac{|\ft u(\xi)|^2}{\sqrt{4\kk^2 + \xi^2}}
(d\xi)_\ld
 \le c_0
 \label{A2}
\end{align}

\noi
for some absolute constant $c_0 > 0$, 
depending on $\kk_0>0$.

\end{lemma}

\noi

The proof of Lemma \ref{LEM:conserved2} is mostly based on  algebraic computations
and thus we omit details.
See the proof of Proposition 4.3 in  \cite{KVZ}.
As in Lemma \ref{LEM:KVZ}, 
the smallness condition \eqref{A2}
 guarantees
term-by-term differentiation of the series \eqref{X1}.
Note that the dependence of $c_0$ on $\kk_0$
comes from the dependence of the implicit constant on $\kk_0$
in Lemma~\ref{LEM:error}.

As in the real line case, Theorem \ref{THM:1} (ii) on $\T$ follows
once we prove the following proposition.

\begin{proposition}\label{PROP:main2}
Let $2\le p<\infty$.
Then, there exists $C= C(p) >0$ such that 
\begin{align*}
\|u(t)\|_{\FL^p(\T)}\le C \big(1+\|u(0)\|_{\FL^p(\T)}\big)^{\frac p2 - 1} \|u(0)\|_{\FL^p(\T)}
\end{align*}

\noi
for any smooth solution $u$ to \eqref{NLS1} on $\T$
and any $t \in \R$.

\end{proposition}

\begin{proof}

Fix $2\leq p < \infty$.
Let $u$ be a global-in-time smooth solution to \eqref{NLS1} on $\T$.
For $\ld \in  \NB $, 
let $u_\ld$ denote the scaled solution to \eqref{NLS1} on 
the dilated torus $\T_\ld$ defined by \eqref{scaling}.
Given small $\eps > 0$ (to be chosen later), 
it follows from  \eqref{Z4} that we can choose sufficiently large $\ld = \ld(\| u(0)\|_{\FL^p(\T)})  \gg 1$ 
such that 
\begin{align}
\|u_\ld(0)\|_{M^{2, p}(\T_\ld)} & \le C \ld^{-\frac{1}{p}}\| u(0)\|_{\FL^p(\T)}
\le \eps \ll 1. 
\label{A3}
\end{align}

\noi
In particular, we may choose 
\begin{align}
\ld \sim  (1+ \| u(0)\|_{ \FL^{p}(\T)})^p.
\label{A3x}
\end{align}

Given $n \in \Z$, 
define $\{u_n\}_{n\in \Z}$ by 
\begin{align*}
u_{\ld, n}(x, t)= 
\mathcal{G}_n(u_\ld)(x, t)
 = e^{-inx} e^{in^2 t } u_\ld ( x - 2n t, t).
\end{align*}

\noi
In view of the Galilean symmetry, 
$u_{\ld, n}$ is the solution to \eqref{NLS1} on $\T_\ld$ with $u_{\ld, n}|_{t = 0} = M_n u_\ld(0)$.
Moreover, we have 
\begin{align}
|\ft u_{\ld, n}(\xi, t)|  = |\ft u_\ld (\xi + n, t)|.
\label{A3a}
\end{align}

\noi
for any $n \in \Z$,  $\xi \in \Z_\ld$,  and $ t\in \R$.
Here, we used the fact that $\ld $ is an integer
such that $\Z \subset \Z_\ld$.

In the following, we fix $\kk  = \frac 12$
and set $\al (u) = \al(\frac 12 ; u)$
and 
\[C_\ld =  \frac{1 + e^{-\pi\ld}} {1 - e^{-\pi\ld}}.\]

\noi
Note that $C_\ld\sim 1$ for $\ld \ge 1$.
From \eqref{X1} with  Lemmas \ref{LEM:error} and \ref{LEM:main1}
and \eqref{A3a}, 
we have
\begin{align*}
\bigg| \al( u_{\ld, n}(t)) 
- C_\ld \int_{\Z_\ld}  \frac{ |\ft u_{\ld, n}(\xi, t)|^2}{1 + \xi^2 } (d\xi)_\ld\bigg|
& \leq  \sum_{j = 2}^\infty
\frac{1}{j}\bigg\| 
(\tfrac 12 - \dx)^{-\frac{1}{2}} \, u_{\ld, n}(t)\,  (\tfrac 12 + \dx)^{-\frac 12}\bigg\|_{\mathfrak{I}_2(\T_\ld)}^{2j}\notag\\
& \les \sum_{j = 2}^\infty
\bigg(\int_{\Z_\ld} 
 \frac{|\ft u_\ld(\xi, t)|^2}{(1 + (\xi - n)^2 )^{\frac{1}{2}-\dl}}(d\xi)_\ld \bigg)^j
\end{align*}

\noi
for any  $\dl > 0$.
By H\"older's inequality, 
we can choose sufficiently small  $\dl = \dl(p) > 0$ such that 
\begin{align}
\int_{\Z_\ld}
 \frac{|\ft u_\ld(\xi, t)|^2}{(1 + (\xi - n)^2 )^{\frac{1}{2}-\dl}}(d\xi)_\ld
& \sim \sum_{k \in \Z} 
\frac{1}{(1 + (k - n)^2 )^{\frac{1}{2}-\dl}} 
\| \ft u_\ld(\xi, 0) \|_{L^2_\xi(I_k\cap \Z_\ld,  (d\xi)_\ld)}^2 \notag\\
& \les  \| u_\ld(0) \|_{M^{2, p}(\T_\ld)}^2
\label{A5}
\end{align}

\noi
uniformly in $n \in \Z$, 
where $I_k = [k - \frac 12, k + \frac12)$, $k \in \Z$, as above.
Then,
in view of  \eqref{A3},  \eqref{A5}, 
and continuity in time, 
we can argue as in the real line case and 
conclude that there exists a time interval $I$ around $t = 0$
such that 
\begin{align*}
\bigg| \al(u_{\ld, n}(t)) 
- C_\ld \int_{\Z_\ld}  \frac{ |\ft u_{\ld, n}(\xi, t))|^2}{1 + \xi^2 } (d\xi)_\ld\bigg|
& \les 
\bigg(\int_{\Z_\ld}  \frac{|\ft u_\ld(\xi, t)|^2}{(1 + (\xi-n)^2 )^{\frac{1}{2}-\dl}} (d\xi)_\ld \bigg)^2
\end{align*}

\noi
for any $t \in I$ and $n \in \Z$, by choosing $\eps > 0$ sufficiently small.
Moreover, 
we may assume that \eqref{A2} is satisfied for all $t \in I$
so that Lemma \ref{LEM:conserved2} is applicable.

Now, compute the $\l^\frac{p}{2}_n$-norm of both sides.
By choosing $\dl = \dl(p) > 0$ sufficiently small, 
Young's inequality yields
\begin{align}
 \bigg\| \al( u_{\ld, n}(t)) 
& - C_\ld \int_{\Z_\ld} \frac{ |\ft u_{\ld, n}(\xi, t))|^2}{1 + \xi^2 } (d\xi)_\ld\bigg\|_{\l^\frac{p}{2}_n}\notag\\
& \les 
\bigg\|\sum_{k \in \Z} 
 \frac{1}{(1 + (k-n)^2 )^{\frac{1}{2}-\dl}}
\| \ft u_\ld (\xi, t)\|_{L^2_\xi(I_k\cap \Z_\ld,  (d\xi)_\ld)}^2\bigg\|_{\l^p_n}^2\notag\\
& \les
 \Big\| \| \ft u_\ld(\xi, t)\|_{L^2_\xi(I_k\cap \Z_\ld,  (d\xi)_\ld)}^2\Big\|_{\l^\frac{p}{2}_n}^2\notag\\
& \sim  \| u_\ld(t)\|_{M^{2, p}(\T_\ld)}^4
\label{A7}
\end{align}

\noi
for any $t \in I$.
Therefore, from  \eqref{Z6}, \eqref{A7}, and 
the conservation of $\al(u_{\ld, n})$, $n \in \Z$, 
\begin{align*}
\| u_\ld(t) \|_{M^{2, p}(\T_\ld)}^2
& \sim 
\|  u_\ld (t) \|_{\MH^{-1, p}(\T_\ld)}^2\notag\\
& \les
\| u_\ld(0) \|_{\MH^{-1, p}(\T_\ld)}^2
+   \| u_\ld(0)\|_{M^{2, p}(\T_\ld)}^4
+   \| u_\ld(t)\|_{M^{2, p}(\T_\ld)}^4\notag\\
& \les
\|u_\ld(0) \|_{M^{2, p}(\T_\ld)}^2
+  \| u_\ld(0)\|_{M^{2, p}(\T_\ld)}^4
+  \| u_\ld(t)\|_{M^{2, p}(\T_\ld)}^4
\end{align*}

\noi
for all $t \in I$.
Hence, by  choosing $\eps > 0$ sufficiently small, 
we conclude from the continuity argument that 
\begin{align*}
\| u_\ld(t) \|_{M^{2, p}(\T_\ld)}
& \les\|u_\ld(0) \|_{M^{2, p}(\T_\ld)}
\end{align*}

\noi
for all $t \in \R$.
Combining this with \eqref{Z3} and \eqref{A3}, 
we obtain
\begin{align}
\| u(t) \|_{\FL^{p}(\T)}
& \les \ld^{\frac{1}{2}-\frac{1}{p}}\|u(0) \|_{\FL^p(\T)}
\label{A8}
\end{align}

\noi
for all $t \in \R$.
Finally, Proposition \ref{PROP:main2} follows from \eqref{A8} with \eqref{A3x}.
\end{proof}

\appendix 
\section{On the modified KdV equation}
\label{SEC:mKdV}

In this appendix, we  consider the 
complex-valued modified KdV equation (mKdV) on $\M = \R$ or $\T$:
\begin{align}
\begin{cases}
 \dt u = - \dx^3 u \pm 6 |u|^2 \dx u, \\
u|_{t= 0} = u_0
\end{cases}
\label{mKdV1}
\end{align}

\noi
and discuss how to derive  the global-in-time bounds \eqref{bd1} and \eqref{bd2}
for Schwartz solutions to the mKdV \eqref{mKdV1}.
The equation \eqref{mKdV1} is known to be completely integrable 
and is closely related to the cubic NLS \eqref{NLS1} \cite{Hirota, SS, KapM, KVZ}.
When the initial data $u_0$ is real-valued, 
the corresponding solution $u$ to \eqref{mKdV1} remains real-valued, 
thus solving
the following real-valued mKdV:
\begin{align}
 \dt u = - \dx^3 u \pm 6 u^2 \dx u.
\label{mKdV2}
\end{align}

\noi
The mKdV enjoys the following scaling symmetry 
\begin{align*}
u(x, t) \longmapsto
u_\ld(x, t) = \ld^{-1} u (\ld^{-1}x, \ld^{-3}t), 
\end{align*}

\noi
inducing the same scaling-critical regularity 
as the cubic NLS \eqref{NLS1}.
In terms of 
the homogeneous Fourier-Lebesgue space $\dot {\FL}^{s, p}(\R)$, 
the critical regularity is given by 
$s_\text{crit}(p) = - \frac 1p$.

The Cauchy problem \eqref{mKdV1} 
has  been extensively studied;
 the complex-valued mKdV \eqref{mKdV1} is known to be locally well-posed
in $H^s(\M)$
for $s \geq \frac 14$ on the real line \cite{KPV93, Tao} and
for $s \geq \frac 13$ on the circle 
\cite{ BO2,  TT, NTT, MPV}.
In the real-valued setting, 
the $I$-method has been applied to 
prove global well-posedness of 
the mKdV \eqref{mKdV2}
in $H^s(\M)$
for $s \geq \frac 14$ on the real line 
and for $s \geq \frac 12$ on the circle 
\cite{CKSTT,  Kishi1}.
See also \cite{CHT, Moli2}
for global existence results in $H^s(\R)$, $s > - \frac {1}{8}$, 
and $L^2(\T)$ in the real-valued setting.
We point out that, 
in the periodic setting, 
the well-posedness studies stated above have been performed on 
the following renormalized mKdV on $\T$:
\begin{align}
\textstyle  \dt u = - \dx^3 u \pm 6 (|u|^2 - \fint_\T |u|^2 dx)\dx u.
\label{mKdV3}
\end{align}

\noi
As in the case of  the cubic NLS \eqref{NLS1} and the renormalized
cubic NLS \eqref{NLS2},
it is easy to see, via the transform:
\[ u(x, t) \longmapsto u(x \pm 6 \mu t, t)\]

\noi
with $\mu = \fint_\T |u|^2 dx$,
that  
the mKdV \eqref{mKdV1} and the renormalized mKdV \eqref{mKdV3}
are equivalent in $L^2(\T)$.
It is, however, known that the renormalized mKdV \eqref{mKdV3} 
behaves better outside $L^2(\T)$.
Indeed, the (unrenormalized) mKdV \eqref{mKdV1} is known to be ill-posed
in negative Sobolev spaces $H^s(\T)$, $s < 0$, \cite{Moli2}
and Fourier-Lebesgue spaces $\FL^p(\T)$, $p > 2$ \cite{KapM}.

In \cite{KVZ}, 
Killip-Vi\c{s}an-Zhang showed that 
Lemma \ref{LEM:KVZ} 
also holds for Schwartz solutions $u$ to the complex-valued mKdV \eqref{mKdV1}.
Therefore, while it is not explicitly stated in \cite{KVZ}, 
their result yields
global well-posedness of the complex-valued
mKdV \eqref{mKdV1}
in $H^s(\M)$
for $s \geq \frac 14$ on the real line  and
for $s \geq \frac 13$ on the circle,
thus matching the known local well-posedness results.

In the real-valued and defocusing case
(i.e.~with the $+$ sign in \eqref{mKdV1}, \eqref{mKdV2}, and \eqref{mKdV3}), 
there are several further results exploiting the completely integrable structure
of the equation.
Before proceeding further, let us define the notion of 
{\it sensible weak solutions}.
See also~\cite{FO, OW5}.

\begin{definition}\label{DEF:sol} \rm
Let $B(\T)$ be a Banach space of functions on $\T$.
Given $u_0 \in B(\T)$, 
we say that
 $u \in C((t_0, t_1); B(\T))$
is a sensible weak solution
to an equation  on $(t_0, t_1)$, $-\infty \leq t_0 < 0 < t_1 \leq\infty$, if,
for any sequence $\{u_{0, n}\}_{n \in \NB}$ of smooth functions
tending to $u_0$ in $B(\T)$, 
the corresponding (classical) solutions $u_n(t)$ with $u_n|_{t = 0} = u_{0, n}$
converge to $u(t)$ in $B(\T)$ for each $t \in (t_0, t_1)$.

\end{definition}

We point that this notion of sensible weak solutions 
is  rather weak.
In particular, sensible weak solutions do not have to satisfy 
the equation even in the distributional sense.
In \cite{KapT}, Kappeler-Topalov 
proved global well-posedness (in the sense of sensible weak solutions) 
of the periodic 
real-valued defocusing mKdV \eqref{mKdV2}
in $L^2(\T)$.
In a recent paper~\cite{KapM}, Kappeler-Molnar 
studied the periodic real-valued defocusing renormalized  mKdV \eqref{mKdV3}
in the Fourier-Lebesgue spaces
and proved, in the sense of sensible weak solutions, 
 local well-posedness and small data global well-posedness
in $\FL^p(\T)$, $2 \le p < \infty$.\footnote{The small data global well-posedness 
also applies to the focusing case.
See Remark on p.\,2217 in \cite{KapM}.}
On the one hand, 
 Molinet \cite{Moli2}  applied the short-time Fourier restriction norm method
and proved that the solutions constructed in \cite{KapT}
are indeed distributional solutions.
On the other hand, 
the solutions  outside $L^2(\T)$
constructed in \cite{KapM} are not yet known to be distributional solutions
at this point.

We now state our result for the mKdV.

\begin{theorem}\label{THM:2}
Let $2 \leq p < \infty$. 

\begin{itemize}
\item[\textup{(i)}] 
The global-in-time bound \eqref{bd1}
in the modulation spaces $M^{2, p}(\R)$
holds for any Schwartz class  solution $u$ to 
 the complex-valued mKdV \eqref{mKdV1} on $\R$.

\smallskip

\item[\textup{(ii)}] 
The global-in-time bound \eqref{bd2}
in the Fourier-Lebesgue spaces $\FL^p(\T)$
holds for 
any smooth solution $u$ to the complex-valued mKdV \eqref{mKdV1} on $\T$.
In particular, 
the same  global-in-time bound \eqref{bd2}
also holds
any smooth solution $u$ to the complex-valued renormalized mKdV \eqref{mKdV3} on $\T$.

\end{itemize}

\end{theorem}

On the circle, 
in view of 
the aforementioned local well-posedness
(in the sense of  sensible weak solutions)
in the Fourier-Lebesgue spaces $\FL^p(\T)$, $2 \leq p < \infty$, in \cite{KapM}
and 
the global-in-time bound \eqref{bd2} in Theorem \ref{THM:2} (ii), 
one may be tempted to conclude
global well-posedness 
(in the sense of  sensible weak solutions) of the real-valued defocusing
renormalized mKdV \eqref{mKdV3}, 
extending globally in time the sensible weak solutions constructed in~\cite{KapM}.
We, however, point out that the construction of the sensible weak solutions
in \cite{KapM} is carried out through the Birkhoff coordinates
and the local existence time is characterized by the openness of the range of the 
Birkhoff map.
In particular, the local existence time in \cite{KapM} is 
not given in terms of the size of initial data in an explicit manner
and hence we do not know  how to conclude such global well-posedness
(in the sense of  sensible weak solutions) 
of the real-valued defocusing
renormalized mKdV \eqref{mKdV3}
in  $\FL^p(\T)$, $2 \leq p < \infty$.
On the real line, there is no known local well-posedness in the modulation spaces $M^{2, p}(\R)$, $p > 2$, 
and hence the global-in-time bound does not lead to any global well-posedness
at this point.
See Remark~\ref{REM:mKdV} 
for a further discussion on this issue.

 Theorem \ref{THM:2} follows from a consideration analogous
 to the proof of Theorem \ref{THM:1} for the cubic NLS 
 but with one important difference.
 On the one hand, the Galilean symmetry~\eqref{Galilei}
 played an important role in the proof of Theorem \ref{THM:1}
for the cubic NLS.
 On the other hand, 
 it is known that the mKdV \eqref{mKdV1} does not enjoy the Galilean symmetry.
Nonetheless, 
if we define  a Galilean transform $\Gf_\be$, $\be \in \R$,
 by 
\begin{align}
u^\beta(x, t) = \Gf_\be(u)(x, t) := e^{- i\be x}e^{2 i\be^3 t}  u (x- 3\be^2 t, t), 
\label{Galilei2}
\end{align}

\noi
then a direct computation shows that if $u$ is a solution to the mKdV \eqref{mKdV1},
then $v = \Gf_\be(u)$ satisfies 
the following mKdV-NLS equation:
\begin{align}
 \dt v = (- \dx^3 v \pm 6 |v|^2 \dx v)
 + 3\be(-i \dx^2 v \pm  2i  |v|^2 v).
\label{mKdV4}
\end{align}

\noi
Then, 
from the conservation of 
$\al(\kk;u)$ in \eqref{X1}
for the mKdV flow and the cubic NLS flow
(\cite[Propositions 4.3 and 4.4]{KVZ}), 
we conclude that $\al(\kk;u)$ is also conserved under the mKdV-NLS equation \eqref{mKdV4}.

\begin{lemma} \label{LEM:KVZ3}
Let $\be \in \R$.
For a Schwartz class  solution $u$ to the mKdV-NLS equation~\eqref{mKdV4}, 
the quantity $\al(\kk; u)$ defined in \eqref{X1} is conserved,
provided that $\kk>0$ is sufficiently large such that 
the smallness condition \eqref{X1a} holds.
\end{lemma}

Once we have Lemma \ref{LEM:KVZ3}
and observe that 
$u^\be = \Gf_\be(u)$ in \eqref{Galilei2} satisfies
\[ 
|\ft{u^\be}(\xi, t)| = 
|\ft u (\xi + \be, t)|\] 

\noi
(compare this with \eqref{Y0a}), 
the global-in-time bounds \eqref{bd1} and \eqref{bd2}
for the mKdV \eqref{mKdV1} follow
exactly in the same manner 
as in the proof of Theorem~\ref{THM:1}.
Hence, we omit details.

\begin{remark}\label{REM:mKdV}\rm

There are  local well-posedness results
by Gr\"unrock \cite{Grun2} and Gr\"unrock-Vega \cite{GV}
on the modified KdV \eqref{mKdV1} 
in the Fourier-Lebesgue spaces on the real line.
In particular, it was shown in \cite{GV} that the mKdV \eqref{mKdV1}
is locally well-posed
in $\FL^{s, p}(\R)$
for $2\leq p < \infty$ and $s \geq \frac{1}{2p}$.
By taking $p \to \infty$, we see that this yields local well-posedness
in almost critical Fourier-Lebesgue spaces.
Unfortunately, our result does not allow us to control the Fourier-Lebesgue norms
on the real line
and thus we do not know how to extend the local-in-time solutions in \cite{Grun2, GV}
globally in time at this point.

In terms of the modulation spaces $M^{2, p}_s(\R)$, 
we recently proved local well-posedness of the mKdV \eqref{mKdV1}
for $s \geq \frac 14$ and $2\leq p < \infty$~\cite{OW2}, 
thus extending the local well-posedness in \cite{KPV93, Tao}.\footnote{See also a recent preprint \cite{CG}
for analogous local  well-posedness of  \eqref{mKdV1}, including $p = \infty$.}
On the one hand, 
$\dot \FL^{\frac 14 , \infty}(\R)$
scales like $\dot H^{-\frac 14}(\R)$
and thus we may say that 
$ M^{2, \infty}_{\frac 14}(\R) $ ``scales like'' 
$\dot H^{-\frac 14}(\R)$ in view of the embedding:
\[ M^{2, p}_{s}(\R) \supset \FL^{s, p}(\R)\]

\noi
for $ p \geq 2$.
On the other hand, 
the $M^{2, p}_{s}(\R)$-norm is weaker
than the $\FL^{s, p}$-norm
and the solution map to the mKdV \eqref{mKdV1}
on $\R$ fails to be locally uniformly continuous
in $M^{2, p}_{s}(\R)$ as soon as $s < \frac 14$.
This is  sharp contrast with the Fourier-Lebesgue case,
where local well-posedness was proved via a contraction argument even for some $s < \frac 14$
\cite{GV}.
Lastly, note that
a slight modification of the proof of Theorem \ref{THM:2} then yields
a global-in-time bound
in  $M^{2, p}_s(\R)$
for $2\le p < \infty$ and $\frac{1}{4} \leq s < 1 - \frac{1}{p}$,
which yields
global well-posedness of the mKdV~\eqref{mKdV1}
in the same range.
See Theorem \ref{THM:3} below.
Combining this global-in-time bound and a persistence-of-regularity
argument, we proved global well-posedness
of the mKdV \eqref{mKdV1}
in  $M^{2, p}_s(\R)$
for 
$s \geq \frac{1}{4}$ and 
$2\le p < \infty$.
See \cite{OW2}
for details.

\end{remark}

\begin{remark}\rm
As in the case of  the cubic NLS, 
the Dirac delta function 
plays a special role for the mKdV.
On the one hand, 
there is 
an existence result for the mKdV \eqref{mKdV1} on $\R$
with the Dirac delta function (with a small multiplicative constant) as initial data \cite{PV}.
On the other hand, via  a scaling analysis, 
one can easily see that continuous dependence
must fail at the Dirac delta function in 
the $\FL^{s, p}(\R)$- and $M^{2, p}_s(\R)$-topologies for $s p < -1$.
See also 
\cite{APH} for an analogous ill-posedness at the Dirac delta function
in the periodic case (in the focusing case).\footnote{While the argument  in \cite{APH} is carried out in $H^s(\T)$, 
it can be easily adapted to the Fourier-Lebesgue setting $\FL^{s, p}(\T)$, $sp < -1$.
We also point out that their result in the defocusing case (Theorem 6.3)
seems to be incorrect.  In particular, the proof of Theorem 6.2 contains an error;
the sech function in the proof needs to be replaced by the csch function,
which causes a breakdown of the proofs of Theorems 6.2 and 6.3.}

\end{remark}

\begin{remark}\label{REM:Rowan}\rm
In the following, we briefly discuss an alternative proof of the global-in-time bounds \eqref{bd1}
and \eqref{bd2} in Theorems \ref{THM:1} and \ref{THM:2}.
This alternative approach   has been brought to our attention by R.\,Killip.
The main idea is to 
consider $\al(\kk; u)$ in \eqref{X1} with a complex number $\kk \in \C$.
Then, we have the following statements.

\begin{lemma}\label{LEM:L1}
For $\kk \in \C$ with $\Re \kk>0$ and $u \in \S(\R)$, we have
\begin{align*}
\Re \tr\big\{(\kk-\dd_x)^{-1} u ( \kk+\dd_x)^{-1} \cj{u} \big\} 
& = \int \frac{2(\Re\kk) |\ft u(\xi + 2\Im \kk)|^2\,d\xi }{4(\Re\kk)^2 + \xi^2},\\
\big\|(\kk-\dd_x)^{-\frac 12} u (\kk+\dd_x)^{-\frac 12} \big\|^2_{\mathfrak I_2(\R)}
& \sim \int \log\big(4 + \tfrac{\xi^2}{(\Re\kk)^2}\big)
\frac{|\ft u(\xi + 2\Im \kk)|^2\,d\xi}{\sqrt{4(\Re\kappa)^2 + \xi^2}}.
\end{align*}

\noi
Moreover, for a Schwartz class  solution $u$ to the cubic NLS \eqref{NLS1} or the mKdV \eqref{mKdV1}, 
the quantity $\al(\kk; u)$ is conserved,
provided that $\Re \kk>0$ is sufficiently large such that 
\begin{align*}
\int_\R  \log\big(4 + \tfrac{\xi^2}{(\Re \kk)^2}\big)\frac{|\ft u(\xi+ 2\Im\kk )|^2}{\sqrt{4 (\Re \kk)^2 + \xi^2 }} d\xi \le c_0
\end{align*}

\noi
for some absolute constant $c_0 > 0$.

\end{lemma}

Here, we stated the results only on the real line but similar statements hold on the circle.
Once we have Lemma \ref{LEM:L1}, 
we may use $\al(\frac 12 + i \frac{n}{2}; u)$ instead of $\al(\frac 12; u_n)$
in \eqref{Y1}, where $u_n$ is defined in \eqref{Y0b}.
In particular, this allows us to proceed 
and establish the global-in-time bounds \eqref{bd1} and \eqref{bd2}
without using the Galilean symmetry \eqref{Galilei}
for the cubic NLS~\eqref{NLS1}
and the Galilean transform \eqref{Galilei2}
for the mKdV \eqref{mKdV1}.

\end{remark}

\section{Controlling the modulation and Fourier-Lebesgue norms of higher regularities}
\label{SEC:B}

In this appendix, we briefly discuss how to derive 
the following global-in-time bounds
on the modulation and Fourier-Lebesgue norms of higher regularities.

\begin{theorem}\label{THM:3}
Let $2 \leq p < \infty$ and $0 \leq s < 1 - \frac 1p$. 

\begin{itemize}
\item[\textup{(i)}] 
There exists $C= C(p) >0$ such that 
\begin{align*}
\|u(t)\|_{M^{2, p}_s(\R)}\le C (1+\|u(0)\|_{M^{2, p}_s(\R)})^{\frac p2 - 1}\|u(0)\|_{M^{2, p}_s(\R)}
\end{align*}

\noi
for any Schwartz class  solution $u$ to the cubic NLS \eqref{NLS1} or the mKdV \eqref{mKdV1}
on $\R$
and any $t \in \R$.

\smallskip

\item[\textup{(ii)}] 
There exists $C= C(p) >0$ such that 
\begin{align*}
\|u(t)\|_{\FL^{s, p}(\T)}\le C \big(1+\|u(0)\|_{\FL^{s, p}(\T)}\big)^{\frac p2 - 1} \|u(0)\|_{\FL^{s, p}(\T)}
\end{align*}

\noi
for any smooth solution $u$ to the cubic NLS \eqref{NLS1} or the mKdV \eqref{mKdV1} on $\T$
and any $t \in \R$.

\end{itemize}

\end{theorem}

One may use a differencing technique as in Section 3 of \cite{KVZ}
to establish global-in-time bounds for higher values of $s$
but we do not pursue it in this paper.
It is worthwhile to note  that when $p = 2$, 
Theorem \ref{THM:3} yields global-in-time control on the
$H^s$-norm of a solution for $0 \leq s <  \frac  12$
{\it without} using a differencing technique.
Compare this with Section 4 of \cite{KVZ},
where their argument (without a differencing technique)
yields  global-in-time control for  $-\frac 12 \leq s < 0$.

In order to prove Theorem \ref{THM:3}, we first introduce  the following modulated Sobolev
space $\MH^{\ta, p}_s(\R)$  with a weight by the norm:
\begin{align*}
\|f\|_{\MH^{\ta, p}_s(\R)}
& = \bigg(\sum_{n \in \Z}\jb{n}^{s p}  \| M_n f\|_{H^\ta}^p \bigg)^\frac{1}{p}\notag\\
& = \bigg(\sum_{n \in \Z}\jb{n}^{s p}   \| \jb{\xi - n}^{\ta} \ft f(\xi)\|_{L^2_\xi}^p \bigg)^\frac{1}{p}, 
\end{align*}

\noi
where $M_n$ is  the modulation operator defined in  \eqref{mod4a}.
On the circle, we define $\MH^{\ta, p}_s(\T)$ in an analogous manner.
Then, we have the following equivalence of the norms.

\begin{lemma}\label{LEM:equiv2}
Let $s \geq 0$ and $\ta + s < -\frac 12$.
Then,  we have
\[ \| f \|_{\MH^{\ta, p}_s} \sim \|f\|_{M^{2, p}_s}\]

\noi
with the understanding that $M^{2, p}_s(\T) = \FL^{s, p}(\T)$ on the circle.

\end{lemma}

The proof of Lemma \ref{LEM:equiv2} is analogous to that of Lemma \ref{LEM:equiv}
and thus we omit details.
In the following, we only consider the cubic NLS~\eqref{NLS1} on the real line case and indicate
where a modification appears in proving Theorem \ref{THM:3}\,(i).
Since the $s = 0$ case is already contained in Theorems~\ref{THM:1}
and~\ref{THM:2}, we restrict our attention to  $s > 0$.

In view of Lemma \ref{LEM:equiv2} with $\theta = -1$, 
we only need to control the following quantity:
\begin{align*}
 \bigg\| \jb{n}^{2s}\al( u_n(t)) 
& - \jb{n}^{2s} \int_\R \frac{ |\ft u_n(\xi, t))|^2}{1 + \xi^2 } d\xi\bigg\|_{\l^\frac{p}{2}_n}
\end{align*}

\noi
for $t \in I$, where $I$ is the time interval used in the proof of Proposition \ref{PROP:main1}.
Proceeding as in \eqref{Y3}, we have 
\begin{align}
 \bigg\| \jb{n}^{2s}\al( u_n(t)) 
& - \jb{n}^{2s} \int_\R \frac{ |\ft u_n(\xi, t))|^2}{1 + \xi^2 } d\xi\bigg\|_{\l^\frac{p}{2}_n}\notag\\
& \les 
\bigg\|\jb{n}^{s} \int_\R \frac{|\ft u(\xi, t)|^2}{\jb{\xi-n}^{1-2\dl}} d\xi\bigg\|_{\l^p_n}^2\notag\\
& \sim
\bigg\|\sum_{k \in \Z} 
 \frac{\jb{n}^{s}}{\jb{k-n}^{1-2\dl}\jb{k}^{2s}}
\cdot \jb{k}^{2s} \| \ft u(\xi, t)\|_{L^2_\xi(I_k)}^2\bigg\|_{\l^p_n}^2
\label{CA2}
\end{align}

\noi
for any $t \in I$.
When $\jb{k} \ges \jb{n}$, 
we apply Young's inequality as in \eqref{Y3} and obtain 
\begin{align}
\text{LHS of }\eqref{CA2}
& \les 
\bigg\|\sum_{k \in \Z} 
 \frac{1}{\jb{k-n}^{1-2\dl}}
\cdot \jb{k}^{2s} \| \ft u(\xi, t)\|_{L^2_\xi(I_k)}^2\bigg\|_{\l^p_n}^2\notag \\
& \les   \| u(t)\|_{M^{2, p}_s}^4
\label{CA3}
\end{align}

\noi
for $s > 0$, 
by choosing $\dl = \dl(p) > 0$ sufficiently small.
When $\jb{k} \ll \jb{n}$, it follows from Young's and H\"older's inequalities:
$\frac 1p + 1 = \frac 1q + (\frac 1r + \frac 2p)$
that 
\begin{align}
\text{LHS of }\eqref{CA2}
& \les 
\bigg\|\sum_{k \in \Z} 
 \frac{1}{\jb{k-n}^{1-2\dl - s}}
\cdot \frac{1}{\jb{k}^{2s}} \cdot \jb{k}^{2s} \| \ft u(\xi, t)\|_{L^2_\xi(I_k)}^2\bigg\|_{\l^p_n}^2\notag \\
& \le 
\bigg\| \frac{1}{\jb{n}^{1-2\dl - s}}\bigg\|_{\l^q_n}^2
\bigg\|\frac{1}{\jb{n}^{2s}}\bigg\|_{\l^r_n}^2
   \| u(t)\|_{M^{2, p}_s}^4\notag\\
& \les 
   \| u(t)\|_{M^{2, p}_s}^4, 
\label{CA4}
\end{align}

\noi
provided that (i) $1-2\dl - s > 1 - \frac 1p - \frac 1r$, 
(ii) $\frac 1r + \frac 2p \leq 1$,
and  (iii) $r > \frac 1{2s}$. 
When $s \le \frac{1}{2} - \frac 1p$, by  choosing $\frac{1}{r} = 2s - $
and $\dl > 0$ sufficiently small, 
we see that all the conditions (i) - (iii) are satisfied.
When $s > \frac{1}{2} - \frac 1p$,
by choosing $\frac 1r = 1 - \frac 2p$
and $\dl > 0$ sufficiently small, 
the conditions (i) - (iii) are satisfied 
for $2\leq p < \infty$ and $\frac 12 - \frac 1p < s < 1 - \frac{1}{p}$. 
Therefore, putting the two cases together, 
we conclude that the estimate \eqref{CA4} holds
for $2\leq p < \infty$ and $0 < s < 1 - \frac{1}{p}$. 

Once we have \eqref{CA3} and \eqref{CA4}, 
we can proceed as in the proof of Proposition~\ref{PROP:main1},
with Lemma~\ref{LEM:equiv2} in place of Lemma~\ref{LEM:equiv}.
The proof for the periodic case follows
in a similar manner.

\begin{ackno}\rm
T.O.~is 
 supported by the European Research Council (grant no.~637995 ``ProbDynDispEq'').
The authors would like to 
thank Rowan Killip for a helpful discussion, 
in particular on Remark \ref{REM:Rowan}.
They also thank Justin Forlano
for pointing out an error in \cite{APH}.

\end{ackno}

\end{document}